\documentclass[a4]{article}   	
\usepackage{hyperref}
\usepackage{amsmath,amssymb,amsthm,bbm}
\usepackage{enumerate}
\usepackage{tikz}[arrows]
\usepackage[textsize=small]{todonotes}
\usetikzlibrary{shapes}
\usepackage{color}
\newcommand{\Z}{\mathbb{Z}}
\newcommand{\mc}[1]{\mathcal{#1}}

\newcommand{\ra}{\rightarrow}
\renewcommand{\P}{\mathbb{P}}

\newcommand{\N}{\mathbb{N}}

\newcommand{\sss}[1]{\scriptscriptstyle{#1}}
\newcommand{\blank}[1]{}

\newcommand{\good}{good }

\newcommand{\NWP}[1]{\boldsymbol{#1}^{\nwarrow}}

\newcommand{\SEP}[1]{\boldsymbol{#1}^{\searrow}}

\newcommand{\nn}{\nonumber}
\newcommand{\rd}{\boldsymbol{\rho}_d}

\newcommand{\bs}[1]{\boldsymbol{#1}}
\newcommand{\bso}{\bs{\omega}}
\newcommand{\bsg}{\bs{\mc{G}}}
\newcommand{\1}{\bs{1}}
\newcommand{\qt}{\Delta}
\newcommand{\rt}{\Delta}
\newcommand{\OLa}{\overline \rt}
\title{Percolation of terraces,\\ and enhancements for the orthant model} 
\author{Mark Holmes\footnote{University of Melbourne.  Email: holmes.m@unimelb.edu.au} \,
 and  Thomas S. Salisbury\footnote{York University. Email: salt@yorku.ca}}

\newcounter{thmcounter}
\addtocounter{thmcounter}{-1}
\newcounter{other}
\newtheorem*{theorem*}{Theorem}
\newtheorem{theorem}[thmcounter]{Theorem}
\newtheorem{lemma}[other]{Lemma}
\newtheorem{corollary}[other]{Corollary}
\newtheorem{proposition}[other]{Proposition}

\theoremstyle{definition}
\newtheorem{definition}[other]{Definition}
\newtheorem{example}[other]{Example}
\newtheorem{remark}[other]{Remark}

\begin{document}
\maketitle

\begin{abstract}
We study a model of an i.i.d.~random environment  in general dimensions $d\ge 2$, where each site is equipped with one of two environments.  The model  comes with a parameter $p$ which governs the frequency of the first environment, and for each dimension $d$ there is a critical parameter $p_c(d)$ at which there is a phase transition for the \emph{geometry} of a particular connected cluster (the cluster is infinite for all $p$).   We use the celebrated methodology of enhancements in this 
novel setting to prove that $p_c(d)$ is strictly monotone in $d$ for this model.  To do so we study the discrete geometry and percolation theory of higher-dimensional structures called terraces. 
\end{abstract}

\noindent{ {\bf Keywords:} percolation, enhancement, random environment, orthant model, critical point, terraces, discrete geometry.}

\medskip

\noindent{ {\bf MSC2020:} 60K35, 52C99}
\tableofcontents

\section{Introduction}
In this paper, we study certain random directed subgraphs of the lattice $\mathbb{Z}^d$ (called the \emph{orthant model} and the \emph{half-orthant} model - see Examples \ref{exa:orthant} and \ref{exa:horthant} below) in which (with probability $p$) a vertex will assume a configuration that blocks immediate passage to sites below it.  This leads to a phase transition in which passage downwards is feasible for small $p$ but not for large $p$.  These random directed graphs have been studied earlier in the case $d=2$, in \cite{DRE} and \cite{DRE2}. More recently, \cite{phase} defined a critical point $p_c(d)$ for these (and other) models, while \cite{shape} examined the asymptotic shape of the set of sites in $\mathbb{Z}^d$ that can be reached from the origin $o$ when $p$ is large.  Beekenkamp \cite{Bk21} extended the validity of the shape theorem for $p$ down to a (potentially different) critical value $p'_c(d)$.  Random walks in various directed graphs including these kinds were studied in \cite{RWDRE} and \cite{RWDRE2}. 

The fact that $p_c(d)$ is monotone increasing in $d$ is straightforward.  
The main result of this paper is that $p_c(d)$ is \emph{strictly} monotone in $d\ge 2$ - see Theorem \ref{thm:p_c}.  This turns out to be highly non-trivial and  requires a detailed excursion into discrete geometry, because the argument centres on the study of percolation for higher dimensional structures called {\it terraces}. 

\subsection{Notation and Results}
Let $d\ge 2$.  Let $[d]=\{1,2,\dots, d\}$, and let $(e_i)_{i \in [d]}$ denote the canonical basis vectors for $\Z^d$.  Set $\mc{E}_+(d)=\{e_i:i \in [d]\}$, $\mc{E}_-(d)=\{-e_i:i \in [d]\}$, and $\mc{E}(d)=\mc{E}_+(d)\cup \mc{E}_-(d)$. When there is no chance of confusion, we will drop $d$ from the above notation.

Let $\mu$ be a probability measure on the measurable space $\big(\mc{P}(\mc{E}),\mc{P}(\mc{P}(\mc{E}))\big)$, where $\mathcal{P}(\mc{A})$ denotes the set of subsets of $\mc{A}$.  We will abuse notation and write $\mu(\mc{G})$ instead of $\mu(\{\mc{G}\})$ where $\mc{G}\subset \mc{E}$ (i.e.~$\mc{G} \in \mc{P}(\mc{E})$).
At each $x\in\mathbb{Z}^d$ independently choose a $\mathcal{G}_x\subset\mathcal{E}$ with law $\mu$. We refer to $\bs{\mc{G}}=(\mathcal{G}_x)_{x\in \mathbb{Z}^d}$ as a \emph{degenerate random environment}. Inserting arrows from $x$ to each of the vertices $\{x+e:e \in \mathcal{G}_x\}$, yields a random directed graph with vertex set $\Z^d$ and directed edge set $\{(x,x+e):x \in \Z^d,e\in \mc{G}_x\}$.
The two examples that are relevant to the present work are as follows.
\begin{example}[Orthant model]
\label{exa:orthant}
$\mu(\mc{E}_+)=p=1-\mu(\mc{E}_-)$.
\end{example}
\begin{example}[Half-orthant model]
\label{exa:horthant}
$\mu(\mc{E}_+)=p=1-\mu(\mc{E})$.
\end{example}
Hereafter we are interested in the set of vertices $\mc{C}_x\subset \Z^d$ that can be reached from $x$ by following the arrows.  In both examples above $\mathcal{C}_x$ \emph{is always infinite}.  To see this note that $e_1\in \mc{E}_+$ and $-e_2\in \mc{E}_-$ so that from $x$ we can follow an infinite self-avoiding path in $\mc{C}_x$ consisting of just $e_1$ and $-e_2$ steps.  Degenerate random environments for which the forward clusters $(\mc{C}_x)_{x \in \Z^d}$ are a.s.~infinite are the natural random directed graphs to study in order to understand {\em random walks in non-elliptic random environments}.  In this context $\mathcal{G}_x$ represents the set of possible steps that the walk can take from $x$ and the fact that forward clusters are a.s.~infinite ensures that the random walker does not get stuck on a finite set of sites (see e.g.~\cite[Lemma 2.2]{RWDRE}).   

 Consider a {\it random walk in random environment} (RWRE) in which the walker chooses uniformly at random from the available arrows in a degenerate random environment $\mathcal{G}_x$. In \cite{RWDRE2} it is shown that this RWRE, for $\mathcal{G}_x$ given by  either Example \ref{exa:orthant} or Example \ref{exa:horthant}, is ballistic in direction  $1:=\sum_{e \in \mc{E}_+}e$ when $d=2$ and $p>p_c(d)$, or when $d\ge 3$ and $p$ is large. In \cite{Bk21} the latter statement is extended to $p>p'_c(d)$.   
For the former example, this should in fact hold for all $p>1/2$ (but we don't know how to prove this), while for the latter it should not be hard to prove that this (and a CLT) holds for all $p>0$ (see e.g.~\cite{BR07}).

We can couple the environments together for all $p$ and all pairs of subsets of $\mc{E}$ in the natural way, using independent uniform random variables $(U_x)_{x \in \Z^d}$. In other words, for $p\in[0,1]$ and $E,F\subset \mc{E}$ define 
$$
\mathcal{G}_x=\mc{G}_x(p)=\begin{cases} E, & \text{if $U_x\le p$}\\
F, & \text{if $U_x>p$.}
\end{cases}
$$
The sets $\mc{C}_x=\mc{C}_x(p)$ are clearly increasing in $E$ and $F$ (so $\mc{C}_x$ for the orthant model is a subset of $\mc{C}_x$ for the half-orthant model). When $E\subset F$ (e.g.~for the half-orthant model $E=\mc{E}_+$ and $F=\mc{E}$), they are also decreasing in $p$.    Let 
\begin{align}
\Omega_1&=\{x \in \Z^d:U_x\le p\}=\{x \in \Z^d:\mc{G}_x=E\}\\
\Omega_0&=\{x \in \Z^d:U_x>p\}=\{x \in \Z^d:\mc{G}_x=F\}.
\end{align}

The orthant model and its phase transition in 2 dimensions are studied in \cite{DRE,DRE2}. Let $o=(0,\dots,0)$ denote the origin in $\Z^d$.  Using a duality with oriented site percolation on the triangular lattice, those papers proved a  shape theorem for a particular boundary of $\mc{C}_o$ for this model, and established improved estimates for the critical value of the dual percolation model. Similar arguments apply to the half-orthant model with $d=2$.  

To formulate the phase transitions in general  dimensions, define for $x\in \Z^d$,
\begin{equation}
L^{(i)}_x(p):=\inf\left\{k\in \Z: x+ke_i\in \mc{C}_o(p)\right\}.
\end{equation}
Also, for $z\in \Z^d$ and $i \in [d]$ define $z_{\{+i\}}=\{z+ke_i:k\ge 0\}$.  For a set $A\subset \Z^d$ define $A_{\{+i\}}=\cup_{z\in A} z_{\{+i\}}$.
The following result concerning the half-orthant model is proved in \cite{phase}.  Related results can be stated for the orthant model, but since e.g.~the orthant model is a non-monotone model (in $p$), it is cleaner to state results for the half-orthant model.
\begin{theorem}[\cite{phase}]
\label{thm:phasetransition}
For the half-orthant model (Example \ref{exa:horthant}) there exists $p_c(d)\in (0,1)$ such that:
\begin{enumerate}[\normalfont(I)]
\item for $p<p_c(d)$, $\mc{C}_o(p)=\Z^d$ a.s.~and;
\item when $p\in (p_c(d),1)$, for each $i \in [d]$
\begin{enumerate}[\normalfont(a)]
\item $L^{(i)}_x(p)$ is a.s.~finite for every $x\in \Z^d$, and 
\item $\mc{C}_o(p)=(\bigcup_{x\in \Z^d}\{x+L^{(i)}_x(p)\})_{\{+i\}}$, and 
\end{enumerate}
\item when $p=p_c(d)$, either $\mc{C}_o(p)=\Z^d$ a.s., or (II)(a) and (b) hold a.s. for each $i\in [d]$.
\end{enumerate}
\end{theorem}
It is shown in \cite{phase} that in any dimension the orthant and half-orthant models share the same $\mc{C}_o(p)$ ``outer''-boundary (for $p\ge 1/2$), in the sense that $L^{(i)}_x(p)$ are equal for the two models for each $x,i$.  Thus, interpreted as a critical point for finiteness of the $L_x^{(i)}$, the orthant and half-orthant models have the same critical value.

 Beekenkamp's $p'_c(d)$ is defined in \cite{Bk21} as the infimum of $p$'s such that $\mc{C}_o(p)$ is contained in certain convex cones. It is known that $p_c(d)\le p'_c(d)$, and it is conjectured that they are equal, but this is only known when $d=2$.

In the next section we will define a geometric class of subsets of $\mathbb{Z}^d$ that we call \emph{terraces}. We can then express the above phase transition in terms of percolation of terraces as follows.
\begin{theorem}
\label{thm:percolationofterraces}
The critical point $p_c(d)$ of the orthant or half-orthant model is the infimum of those $p$ such that there a.s.~exists a terrace $\Lambda$ with $\Lambda \subset \Omega_1$.
\end{theorem}
This paper is therefore a contribution to the general program of extending standard percolation techniques to such higher-dimensional structures. For other examples of high-dimensional percolation structures, see \cite{GH10, GHK14}.  Our main result is the following.
\begin{theorem}
\label{thm:p_c}
For the orthant or half-orthant model  $p_c(d) < p_c(d+1)$.
\end{theorem}
We will prove Theorem \ref{thm:p_c} via the methodology of enhancements, introduced by Aizenman and Grimmett in \cite{AG91}.  This is the ``standard'' methodology for proving strict inequality results for critical points in percolation models (see e.g.~\cite{MS2018,BBR14}).   Application of this methodology usually requires verifying some combinatorial statement.  Sometimes the statement in question is straightforward (or even obvious), while elsewhere it can be extremely challenging (see e.g.~\cite{GS98}, where strict inequality between the critical values of bond and site percolation is established in various lattices).    The full story of enhancements of site percolation on $\Z^d$ is complicated and currently
incomplete (see e.g.~\cite{BBR14}).

Our use of enhancements involves the deformation of terraces, which poses a significant challenge. 
Sections \ref{sec:terraces}--\ref{sec:localterr} are entirely geometric, establishing deterministic facts about terraces that will be required in Section \ref{sec:enhancement}, where the main theorem is proved. The approach in Section  \ref{sec:enhancement} is to use a slab to interpolate between $\mathbb{Z}^d$ and $\mathbb{Z}^{d+1}$. We analyze this slab by perturbing the $\mathbb{Z}^d$ model at points $V_d$ of a sub-lattice, which is taken to be sparse in order to ensure sufficient independence. The enhancement approach then requires us to estimate the number of pivotal sites in $V_d$. 

We will do that by modifying a given terrace locally in order to pass through some point of $V_d$. The obstacle to achieving this is that our terraces can be tightly constrained by boundary conditions, which can in some cases impose a high degree of rigidity on the terraces allowed. A consequence of the ideas of Section \ref{sec2intro}, that illustrates this issue, is that the kinds of boundary configurations that make the terrace rigid must already force it to go through points of $V_d$. Our approach to this problem is to show that the terraces already contain (or can be deformed locally in order to contain) long linear segments.

\begin{remark}
An alternate approach which would also establish our main result (while bypassing the interpolation via slabs) would be to fix the dimension $d$ and to consider two models, one with sets $E_0,F_0$ and the other with sets $E_1,F_1$, satisfying the following:
\begin{itemize}
\item $F_0=\mc{E}$ and $\{e_1\}\subset E_0\subsetneq \mc{E}_+$, and 
\item $F_1=\mc{E}$ and $E_0\subsetneq E_1\subset \mc{E}_+$,
\end{itemize}
Strict monotonicity would follow from establishing that the critical points for the above two models are not equal.  However, this approach seems to require an understanding of a more general class of geometric objects than the  terraces considered here.
\end{remark}

\section{Terraces}
\label{sec2intro}
As indicated earlier, we will require a detailed study of the geometry of certain sets that we call \emph{terraces}.

For $z\in \Z^d$  we let $z^{\sss[i]}=z\cdot e_i$ (note that when used in a \emph{superscript}, the notation  $[i]$  refers to the $i$-th coordinate of $z$ not $z^{\{1,\dots, d\}}$) and define
\begin{align}
z_+&=\{y\in \Z^d:y^{\sss[i]}\ge z^{\sss[i]} \text{ for each }i \in [d]\}, \text{ and }\nn\\
z_-&=\{y\in \Z^d:y^{\sss[i]}\le z^{\sss[i]} \text{ for each }i \in [d]\}.\nn
\end{align}
For $C\subset \Z^d$ we define $-C=\{-x:x\in C\}$,    
\begin{align}
C_+=\bigcup_{z\in C}z_+, \quad \text{ and }\quad C_-=\bigcup_{z\in C}z_-.\label{C_+def}
\end{align}
Trivially $\varnothing_+=\varnothing$ and
\begin{equation}
\label{+subset}
\text{if $A\subset B$ then $A_+\subset B_+$.}
\end{equation}

\begin{itemize}
\item A set $C\subset\mathbb{Z}^d$ is {\it saturated} if for every $y\in \mathbb{Z}^d$ and every  $i\in [d]$, there exist $ z\in C$ and $k\in \mathbb{Z}$ such that $z=y+ke_i$. 
\item A set $C\subset\mathbb{Z}^d$ is {\it solid above} (resp.~\emph{ below}) if $C_+=C$ (resp.~$C_-=C$).
\item For $i \in [d]$, a set $C$ is $i$-{\em bounded below} (resp.~\emph{above}) if for every $y\in \Z^d$, the set $\{k \in \Z:y+ke_i\in C\}$ is bounded below (resp.~above).  
\item We say that $C$ is {\em bounded below} (resp.~\emph{above}) if it is $i$-bounded below (resp.~above) for every $i\in [d]$.
\item If $C$ is saturated and bounded below, we may define $\gamma_{\sss C}:\mathbb{Z}^d\times [d]\to C$ by  $\gamma_{\sss C}(y,i)=z$, where $z=y+ke_i$ for $k$ the smallest integer making this point in $ C$.  For such a $C$, let $\lambda_C=\{\gamma_{\sss C}(y,i): y\in \mathbb{Z}^d, i\in [d]\}$.
\end{itemize}
It is trivial that $B_+$ is solid above and $B_-$ is solid below for any set $B\subset \Z^d$.  
The following are also trivial observations:
\begin{equation}
\label{solidabove_union}
\text{ if $G$ and $G'$ are solid above then so is $G\cup G'$, }
\end{equation}
\begin{equation}
\label{boundedbelow_union}
\text{ if $G$ and $G'$ are bounded below then so is $G\cup G'$, }
\end{equation}
\begin{definition}[Good set]
\label{def:good}
A set $C\subset \Z^d$ is said to be {\em \good} if $C$ is saturated, solid above, and bounded below.  
\end{definition}
It follows from the above that 
\begin{equation}
\label{good_union}
\text{ if $G$ and $G'$ are good then so is $G\cup G'$. }
\end{equation}

\begin{definition}[Terrace]
\label{def:terrace}
A set $\Lambda\subset \Z^d$ is a {\it terrace} if $\Lambda=\lambda_C$ for some good set $C$.  
\end{definition}

The following is essentially a reformulation of Theorem \ref{thm:phasetransition}, in the language of terraces, though we will not prove this until Section \ref{sec:terracesandCo}.
\begin{proposition} 
\label{prop:CoStructure}
For the half-orthant model, Example \ref{exa:horthant}:
\begin{enumerate}
\item[\normalfont(i)] $\mc{C}_o$ is solid above for each $p\in [0,1]$ and is saturated for $p<1$; 
\item[\normalfont(ii)]  If $p\in (p_c,1)$ then $\mc{C}_o$ is good, $\gamma_{\sss{\mc{C}_o}}(x,i)=x+L_x^{(i)}$, and 
$\lambda_{\mc{C}_o}=\{\gamma_{\sss{\mc{C}_o}}(x,i): x\in \mathbb{Z}^d, i\in [d]\}\subset \mc{C}_o$ is a terrace with $\lambda_{\mc{C}_o}\subset \Omega_1$ and $(\lambda_{\mc{C}_o})_+=\mc{C}_o$.
\item[\normalfont(iii)] If $p<p_c$ then $\mc{C}_o(p)=\Z^d$, and $\Omega_1$ contains no terrace.
\end{enumerate}
\end{proposition}

\begin{proof}[Proof of Theorem \ref{thm:percolationofterraces}]
The statement of Theorem \ref{thm:percolationofterraces} for the half-orthant model follows immediately from Proposition \ref{prop:CoStructure}. The corresponding statement for the orthant model follows in turn, using results of \cite{phase} which establish that $(\mc{C}_o(p))_+$ for the orthant model coincides with $\mc{C}_o(p)$ for the half-orthant model. 
\end{proof}

\begin{figure}
\includegraphics[scale=0.6]{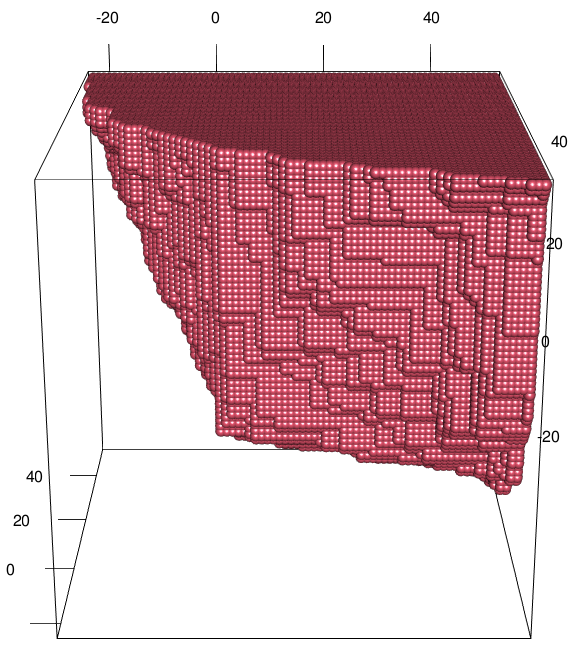}
\caption{A simulation of a portion of $\mc{C}_o$ for the half-orthant model of Example \ref{exa:horthant} in 3 dimensions when $p=0.95$.  The structure of the terrace $\lambda_{\mc{C}_o}$ is clearly visible.}
\label{fig:3d_1} 
\end{figure}
Figure \ref{fig:3d_1} depicts a realisation of part of the boundary of $\mc{C}_o$  for the orthant model in 3 dimensions with  $p=0.95$.

\subsection{Facts about terraces}
\label{sec:terraces}
This section consists of deterministic geometric properties about terraces.  For $a,b \in \Z^d$ with $b\in a_+$  define the box $[a,b]$ with corners $a$ and $b$ by 
\begin{equation}
[a,b]=\{z\in \Z^d: a^{\sss[i]}\le z^{\sss[i]} \le b^{\sss[i]} \text{ for each }i \in [d]\}.\label{boxdef}
\end{equation}
 Note that $[a,b]$ is non-empty, but consists of a single point in the case $a=b$.
\begin{lemma}
\label{lem:terrace}
If $\Lambda\subset\mathbb{Z}^d$ is a terrace then:
\begin{enumerate}[\normalfont(i)]
\item $\Lambda$ is saturated, bounded below, and for each $z\in\Lambda$ there is an $i\in [d]$ with $z-e_i\notin \Lambda_+$;
\item If $C$ is good and such that $\Lambda=\lambda_C$, then $C=\Lambda_+$ (and hence $\Lambda_+$ is good);
\item For any line $L$ parallel to an axis, $\Lambda\cap L$ is finite;
\item $\Lambda_+\cap \Lambda_-=\Lambda$;
\item $(\Lambda_+)^c$ is solid below, and $-((\Lambda_+)^c\cup \Lambda)$ is good; 
\item $\Lambda_+\setminus\Lambda$ is solid above;
\item If a nearest neighbour path starts in $\Lambda_+$ and ends in $(\Lambda_+)^c$, and if $k$ is the first time that the path is in $(\Lambda_+)^c$ then the path is in $\Lambda$ at time $k-1$;
\item $\Lambda_+\cap (x_-)$ is finite for every $x$.
\item If $x\in\Lambda_+$, $y\in\Lambda$, and $y\in x_+$ then $[x,y]\subset\Lambda$. 
\end{enumerate}
\end{lemma}

\begin{proof}[Proof of Lemma \ref{lem:terrace}]
 Suppose that $\Lambda$ is a terrace.  Then $\Lambda=\lambda_C$ for some good set $C$.  

Let $y\in \Z^d$ and $i \in [d]$. Then $\gamma_{\sss C}(y,i)\in \Lambda$, so $\Lambda$ is saturated.  Next, $\Lambda\subset C$ and $C$ is bounded below so $\Lambda$ is also bounded below.    
Now let $z\in \Lambda$.  Then $z=\gamma_{\sss C}(y,i)$ for some $y\in \Z^d$ and $i \in [d]$, so $z-e_i\notin C$.  Since $C=C_+$ (solid above), $z-e_i\notin C_+$ and therefore $z-e_i \notin \Lambda_+$ by \eqref{+subset}.  This verifies (i).

For (ii), note that by \eqref{+subset}, $\Lambda_+\subset C_+$.  Since $C$ is solid above, $\Lambda_+\subset C$.  Next, if $z\in C$ and $i \in [d]$ then $z':=\gamma_{\sss C}(z,i)\in \Lambda$ and $z\in z'_+$, so $z \in \Lambda_+$.  Thus, $C\subset\Lambda_+$.

For (iii), suppose without loss of generality that $L$ is the line $\Z e_1$, and suppose that this line contains infinitely many points in $\Lambda$.  Then (since $\Lambda$ is bounded below) there exists some $j\in [d]$ and an increasing sequence $(n_r)_{r \in \N}$ such that $n_r e_1=\gamma_{\sss{\Lambda_+}}(n_re_1,j)$.  Now consider the line $L'=\Z e_1-e_j$.  Since $\Lambda$ is saturated, there exists $x \in L'\cap \Lambda\subset \Lambda_+$, and since $\Lambda_+$ is solid above, $x+ne_1\in \Lambda_+$ for every $n \in\N$.  Choose $n_r$ such that $n_r>x\cdot e_1$.  Then $n_r e_1=\gamma_{\sss{\Lambda_+}}(n_re_1,j)$ but $n_re_1-e_j\in \Lambda_+$, which is a contradiction.    

To prove (iv), it is trivial that $\Lambda \subset \Lambda_+\cap \Lambda_-$.  To prove the reverse inclusion, suppose that $x\in \Lambda_-$, so that $x\in z_-$ for some $z\in \Lambda$.  Then by the definition of a terrace and (ii), for some $j\in [d]$ we have $z-e_j \notin \Lambda_+$, and thus $x-e_j \notin \Lambda_+$.  Suppose also that $x\in \Lambda_+\setminus \Lambda$.  Then for each $i\in [d]$, $\gamma_{\sss{\Lambda_+}}(x,i)\ne x$, so there exists $n_i\in \N$ such that $x-n_ie_i\in \Lambda_+$, and since $\Lambda_+$ is solid above we have that $x-e_i\in \Lambda_+$ for each $i$. Thus $x-e_j\in \Lambda_+$, which is a contradiction.

For (v), let $z\in (\Lambda_+)^c$, and let $x\in z_-$.  Then $z \in x_+$, so if $x \in \Lambda_+$ then $z \in \Lambda_+$ as well.  This is a contradiction, hence $x \in (\Lambda_+)^c$, so $(\Lambda_+)^c$ is solid below.   For the second claim it suffices to show that $(\Lambda_+)^c\cup \Lambda$ is saturated, bounded above and solid below.  Note that $(\Lambda_+)^c\cup \Lambda$ is saturated since $\Lambda$ is (by (i)).  Let $x\in \Z^d$ and $i\in [d]$.  Then $\gamma_{\sss{\Lambda_+}}(x,i)\in \Lambda$  and $\{\gamma_{\sss{\Lambda_+}}(x,i)+ke_i:k\ge 0\}\subset \Lambda_+$ since $\Lambda_+$ is solid above.  Thus, $\{k \in \Z:x+ke_i \in (\Lambda_+)^c\}$ is bounded above.  By (iii) we also have that $\{k \in \Z:x+ke_i \in \Lambda\}$ is bounded above.  Thus $(\Lambda_+)^c\cup \Lambda$ is bounded above.
Suppose that $(\Lambda_+)^c\cup \Lambda$ is not solid below.  Then there exist $x\in (\Lambda_+)^c\cup \Lambda$ and $y \in x_-$ such that $y \notin (\Lambda_+)^c\cup \Lambda$ (i.e.~$y \in \Lambda_+\setminus \Lambda$).  Since $(\Lambda_+)^c$ is solid below we must have that $x\in \Lambda$.  Since $y \in x_-$ this means that $y\in \Lambda_-$.  Thus $y \in (\Lambda_+\setminus \Lambda) \cap (\Lambda_-)$, which contradicts (iv). 

For (vi), let $x\in\Lambda_+\setminus \Lambda$.  Then  $x-e_i\in \Lambda_+$ for every $i \in [d]$.  Let $y\in x_+$, so $y\in\Lambda_+$. If $y\in\Lambda$ then there is a $j$ with $y-e_j\notin\Lambda_+$. But then $x-e_j\notin\Lambda_+$ as well (as $(x-e_j)^{[i]}\le (y-e_j)^{[i]}$ for every $i$).  This gives a contradiction, so $y\in \Lambda_+\setminus \Lambda$ for every $y \in x_+$, as required.

For (vii), there is a first time $k$ the path enters $(\Lambda_+)^c$, so at $k-1$ it is at a point $x\in \Lambda_+$. Since $x+e_i\in \Lambda_+$ for each $i \in [d]$ the $k$th step must be from $x$ to some $x-e_j\in (\Lambda_+)^c$.  Since $x-e_j\in \Lambda_+^c$ we have that $\gamma_{\sss{\Lambda_+}}(x,j)=x$, so $x\in \lambda_{\Lambda_+}=\Lambda$.

Turning to (viii), if $\Lambda_+\cap x_-$ is infinite then there is an $i$ such that the set of reals $e_i\cdot (\Lambda_+\cap x_-)$ fails to be bounded below.   Let $n \in \N$.  Then there exists $y\in \Lambda_+\cap x_-$ with $y^{\sss[i]}<x^{\sss[i]}-n$, and since $\Lambda_+$ is solid above it follows that $x-ne_i\in y_+\subset\Lambda_+$.  Having this for every  $n\in\N$ contradicts the assumption that $\Lambda_+$ is bounded below.

To see (ix), note first that since $y \in x_+$, $[x,y]\subset x_+$ by definition.  Let $z\in[x,y]$.  Then $z\in x_+\subset\Lambda_+$, and $y \in z_+$.  If $z\in\Lambda_+\setminus\Lambda$, then (vi) would imply that $y\in z_+\subset\Lambda_+\setminus\Lambda$, which is impossible. So $z\in\Lambda$. 
\end{proof}
\begin{remark}
\label{rem:terrace_neighbour}
It follows from Lemma \ref{lem:terrace}(i),(vii) that if $\Lambda$ is a terrace then
\[ x\in \Lambda \iff x\in \Lambda_+ \text{ and }x-e_i\notin \Lambda_+ \text{ for some }i.\]
\end{remark}
\begin{lemma}
\label{lem:decreasing_terraces}
If $\Lambda^{\sss(0)}\supset \Lambda^{\sss(1)}\supset \dots$ is a decreasing sequence of terraces, then $\hat{\Lambda}:=\cap_{s=0}^\infty \Lambda^{\sss(s)}$ satisfies $\hat{\Lambda}=\lambda_G$, where $G:=\cap_{s=0}^\infty \Lambda^{\sss(s)}_+$ is good.  
\end{lemma}
\begin{proof}
Clearly $G$ is bounded below since $\Lambda^{\sss(0)}_+$ is.  If $x\in G$ and $y \in x_+$ then $x\in \Lambda^{\sss(s)}_+$ for every $s$, and since $\Lambda^{\sss(s)}_+$ is solid above we have $y \in \Lambda^{\sss(s)}_+$ (for each $s$).  Thus $G$ is solid above.  To show that $G$ is saturated, it is sufficient to show that $\hat{\Lambda}$ is saturated (since $G \supset \hat{\Lambda}$).  Let $L$ be a line parallel to one of the axes.  Then $L\cap \Lambda^{\sss(s)}$ is finite, non-empty, and the sets $L\cap \Lambda^{\sss(s)}$ are decreasing in $s$.  Hence $\cap_{s=0}^\infty (L\cap \Lambda^{\sss(s)})=L \cap \hat{\Lambda}$ is non-empty as required.  We have shown that $G$ is good.  

To show that $\hat{\Lambda}=\lambda_G$, note that the $\Lambda^{\sss(s)}_+$ are decreasing as well.  Suppose that $x\in \cap_{s=0}^\infty \Lambda^{\sss(s)}$.  Then $x\in G$ and 
$x=\gamma_{\Lambda^{\sss(0)}_+}(x,i)$ for some $i \in [d]$, so $x-e_i \notin \Lambda^{\sss(0)}_+$ and hence $x-e_i \notin G$.
Therefore $x=\gamma_{G}(x,i)\in \lambda_G$.  

Conversely, suppose that $x\in \lambda_G$.  Then $x\in G$ and $x=\gamma_{G}(x,i)$ for some $i \in [d]$, so $x-e_i\notin G$.  Therefore $x \in \Lambda^{\sss(s)}_+$ for every $s$, and there exists $s_x$ such that for every $s>s_x$, $x-e_i \notin \Lambda^{\sss(s)}_+$.  Therefore $x\in\Lambda^{\sss(s)}$ for every $s>s_x$.  But the $\Lambda^{\sss(s)}$ are decreasing, so $x\in \Lambda^{\sss(s)}$ for every $s$, so $x\in \cap_{s=0}^\infty \Lambda^{\sss(s)}$.
\end{proof}

If $\Lambda$ is a terrace then define 
\begin{equation*}
H_{\Lambda}=\{x \in \Lambda: x+e_i \in \Lambda \text{ for every }i \in [d]\}.
\end{equation*}  
A terrace $\Lambda\subset \Z^d$ is said to be {\em irreducible} if $H_\Lambda$ is empty.  

\begin{lemma}
\label{lem:minus_terrace}
Let $\Lambda$ be a terrace.  Then $-\Lambda$ is a terrace if and only if $H_\Lambda=\varnothing$.
\end{lemma}
\begin{proof}
Suppose that $x\in H_\Lambda$.  Then $x\in \Lambda$ and $x+e_i\in \Lambda$ for every $i \in [d]$.  Thus $-x\in -\Lambda$ and $-x-e_i\in -\Lambda$ for every $i \in [d]$.  Thus by Lemma \ref{lem:terrace}(i), $-\Lambda$ is not a terrace.

Now suppose that $H_\Lambda$ is empty.  By Lemma \ref{lem:terrace}(v) we have that $K:=-((\Lambda_+)^c\cup \Lambda)$ is good.  We claim that $-\Lambda=\lambda_K$.  Let $u \in -\Lambda \subset K$.  Then $v:=-u\in \Lambda$ and since $H_\Lambda$ is empty there exists $i \in [d]$ such that $v+e_i\in \Lambda_+\setminus \Lambda$.  Therefore $v+e_i\notin (\Lambda_+)^c\cup \Lambda=-K$ and thus $u-e_i \notin K$.  Therefore $u=\gamma_{\sss K}(u,i)\in \lambda_K$.  Thus, $-\Lambda \subset \lambda_K$.  Now let $u \in \lambda_K$.  Then $u \in K$ and there exists $i\in [d]$ such that $u-e_i\notin K$.  Thus, $v=-u\in (\Lambda_+)^c\cup \Lambda$ and $v+e_i \notin (\Lambda_+)^c\cup \Lambda$.  Then $v+e_i \in \Lambda_+\setminus \Lambda$ which implies that $v=(v+e_i)-e_i \in \Lambda_+$.  Since also $v=\in (\Lambda_+)^c\cup \Lambda$ we conclude that $v\in \Lambda$ and therefore $u \in -\Lambda$.  This confirms that $\lambda_K\subset -\Lambda$.
\end{proof}

We'll only require a special case of the following construction (see Lemma \ref{specialpath}), but give the general version here for completeness.
 \begin{lemma}
\label{standardpath}
 Let $\Lambda$ be a terrace, and suppose $x,y\in\Lambda$. There is a path $z_{\sss(0)}, \dots,z_{\sss(n)}$ of points in $\Lambda$ such that 
 \begin{enumerate}[\normalfont(a)]
\item $z_{\sss(0)}=x$, $z_{\sss(n)}=y$;
\item the $\ell_1$ distance from $x$ to $z_{\sss(k)}$ is strictly increasing in $k$, and the $\ell_1$ distance from $z_{\sss(k)}$ to $y$ is strictly decreasing;
\item each increment $z_{\sss(j+1)}-z_{\sss(j)}$ has the form $e_i$, $-e_k$, or $e_i-e_k$ for some $i\neq k$; 
\item each coordinate $z_{\sss(j)}^{\sss[i]}$ is monotone in $j$.
\end{enumerate}

\end{lemma}
\begin{proof} Property (d) follows from the others, as they together imply that the $e_i$ or $e_i-e_k$ are only used when $x^{\sss[i]}<z^{\sss[i]}<y^{\sss[i]}$ and the $-e_k$ or $e_i-e_k$ are only used when $x^{\sss[k]}>z^{\sss[k]}>y^{\sss[k]}$. We therefore will ignore (d) in the rest of the proof.

By Lemma \ref{lem:terrace}(i) we cannot have $x^{\sss[i]}<y^{\sss[i]}$ for every $i$, nor can we have $x^{\sss[i]}>y^{\sss[i]}$ for every $i$.  If $x^{\sss[i]}\le y^{\sss[i]}$ for every $i$, fix an $i$ with $x^{\sss[i]}<y^{\sss[i]}$ (if no such $i$ exists then $x=y$ and we are done) and let $z_{\sss(1)}=x+e_i\in\Lambda_+$.  Then $z^{\sss[j]}_{\sss(1)}\le y^{\sss[j]}$ for every $j$, so by Lemma \ref{lem:terrace}(vi) we see that $z_{\sss(1)}\in \Lambda$.  Iterating this argument, we keep stepping in direction $e_{i}$ till the $i$-coordinates agree. Then pick another coordinate to increase, etc. Eventually all coordinates will agree, and we've reached $y$. 

If $x^{\sss[i]}\ge y^{\sss[i]}$ for every $i$, construct a path from $y$ to $x$ as above, and traverse it backwards. 

Now assume that $x^{\sss[i]}<y^{\sss[i]}$ and $x^{\sss[k]}>y^{\sss[k]}$. Define
\begin{equation}
z_{\sss(1)}=\begin{cases}
x-e_k, &\text{ if }x-e_k\in \Lambda_+\\
x-e_k+e_i, & \text{ if }x-e_k\notin \Lambda_+, \, x-e_k+e_i\in\Lambda_+\\
x+e_i, & \text{ otherwise}.
\end{cases}
\end{equation}
If $x-e_k\in  \Lambda_+$ then $z_{\sss(1)}:=x-e_k\in \Lambda$ by Lemma \ref{lem:terrace}(v).  If $x-e_k\notin\Lambda_+$ and $x-e_k+e_i\in\Lambda_+$, then $z_{\sss(1)}:=x-e_k+e_i=\gamma_{\sss{\Lambda_+}}(x-e_k,i)\in \Lambda$.  Finally, if $x-e_k+e_i\notin\Lambda_+$ (and $x\in \Lambda_+$), then $x+e_i\in \Lambda_+$ and $z_{\sss(1)}:=x+e_i=\gamma_{\sss{\Lambda_+}}(x+e_i,k)\in\Lambda$, as required. Either this reaches $y$, or we can repeat one or the other of the above arguments, to find $z_{\sss(2)}$, and then $z_{\sss(3)}$, etc.
\end{proof}

 \begin{lemma}
\label{lem:connected}
Let $\Lambda$ be a terrace.  If $x,y\in\mathbb{Z}^d\setminus \Lambda_+$ or $x,y\in\Lambda_+\setminus\Lambda$ then there is a self avoiding nearest neighbour path connecting $x$ and $y$ which is disjoint from $\Lambda$. \end{lemma}
\begin{proof} Suppose $x,y\in\mathbb{Z}^d\setminus \Lambda_+$. Set $z_{\sss(0)}=x$. Given $z_{\sss(k)}$, take $z_{\sss(k+1)}=z_{\sss(k)}-e_i$ as long as we can find an $i$ with $z_{\sss(k)}^{\sss[i]}>y^{\sss[i]}$. Because $x\in (z_{\sss(k)})_+$ we have $z_{\sss(k)}\in\mathbb{Z}^d\setminus \Lambda_+$. Once we've run out of such $i$, start taking $z_{\sss(k+1)}=z_{\sss(k)}+e_j$ where $z_{\sss(k)}^{\sss[j]}<y^{\sss[j]}$. Because $y\in(z_{\sss(k)})_+$ we still have $z_{\sss(k)}\in\mathbb{Z}^d\setminus \Lambda_+$. Eventually this sequence reaches $y$. 

The same argument applies if $x,y\in\Lambda_+\setminus\Lambda$, except first we increase coordinates and then we decrease them. \end{proof}

\subsection{Corners}
A point $z$ is a {\it corner} of $C\subset \Z^d$ if $z\in C$ but for each $i$, $z-e_i\notin C$.  
In this section, we examine corners of terraces. We will really use local versions of these results (stated in the next section), but we will prove the terrace versions first, as they are both simpler and of more intrinsic interest.

\begin{lemma}
  \label{pushingup}
 Let $\Lambda$ be a terrace, with a corner at $z$.   Then:
 \begin{itemize}
\item[\normalfont(i)] $\Lambda_+\setminus\{z\}$ is good;
\item[\normalfont(ii)]  $\Lambda\setminus\{z\} \subset \lambda_{\Lambda_+\setminus\{z\}}\subset (\Lambda\setminus\{z\}) \cup \{z+e_j:j \in [d]\}$.
\end{itemize}
 \end{lemma}
 \begin{proof}
$\Lambda_+\setminus\{z\}$ is bounded below because $\Lambda_+$ is, and it is saturated because $\Lambda_+$ was saturated and $\Lambda_+\setminus\{z\}$ contains $z+e_j$ for each $j \in [d]$. So let us prove that $\Lambda_+\setminus\{z\}$ is solid above.

Let $x\in \Lambda_+\setminus\{z\}$ and $y \in x_+$.  Then $y \in \Lambda_+$, and we need only show that $y\neq z$. So assume $z=y$. Then $z\in x_+$ and $x\neq z$, so we have $z\in (x+e_i)_+$ for some $i$. Hence $z-e_i \in x_+\subset \Lambda_+$.  But $z-e_i\notin\Lambda$, because $z$ is a corner, so $z-e_i \in \Lambda_+\setminus \Lambda$.   By Lemma \ref{lem:terrace} (vi) we conclude that $z\in \Lambda_+\setminus \Lambda$ which is a contradiction. So $z\neq y$ as required. 
 
 To prove (ii), note that if $x\in \Lambda \setminus \{z\}$ then $x=\gamma_{\sss{\Lambda_+}}(x,i)$ for some $i$.  Thus $x=\gamma_{\Lambda_+\setminus \{z\}}(x,i)$ as well, so $x\in \lambda_{\Lambda_+\setminus\{z\}}$.  This proves the first half of (ii).  For the second half, let  $x\in \lambda_{\Lambda_+\setminus\{z\}}$. Then $x\in\Lambda_+\setminus\{z\}$ and $x-e_j\notin\Lambda_+\setminus\{z\}$ for some $j$. The latter implies that either $x-e_j\notin\Lambda_+$ or $x-e_j=z$. In the first case, $x\in\Lambda\setminus\{z\}$ and in the second, $x=z+e_j$. Either way, (ii) is shown.
 \end{proof}
 For a terrace $\Lambda$ containing a corner $z\in \Lambda$, define 
 \begin{equation}
 \Lambda_{(\uparrow z)}=\lambda_{\Lambda_+\setminus\{z\}}.
 \end{equation}
 We describe $ \Lambda_{(\uparrow z)}$ as the terrace resulting from {\em pushing up} the corner $z$ of the terrace $\Lambda$. See e.g.~Figure \ref{fig:pushed_up}. 

\begin{figure}[h]
\centering
\begin{tikzpicture}
\draw (0,3) node[circle, fill=black]  {};
\draw (1,3) node[circle, fill=black]  {};
\draw (2,2) node[circle, fill=black]  {};
\draw (2,1) node[circle, fill=black]  {};
\draw (3,1) node[circle, fill=black]  {};
\draw (4,0) node[circle, fill=black] {};
\draw (5,0) node[circle, fill=black] {};
\draw (1+.35,3+.35) node {$ $};
\draw (2-.4,1) node {$x$};
\draw (4-.4,0) node {$z$};
\draw[thick, black] (0,3) -- (1,3);
\draw[thick, black] (1,3) -- (2,2);
\draw[thick, black] (2,2) -- (2,1);
\draw[thick, black] (2,1) -- (3,1);
\draw[thick, black] (3,1) -- (4,0);
\draw[thick, black] (4,0) -- (5,0);
\draw[thick, dashed, black] (2,2) -- (3,1);
\draw[thick, dashed, black] (3,1) -- (4,1) -- (5,0);
\end{tikzpicture}
\caption{Part of a terrace $\Lambda$ with $d=2$ in which $x\in H_\Lambda$ and $z$ are corners.  The dashed lines show the terrace with both $x$ and $z$ pushed up.}
\label{fig:pushed_up}
\end{figure}
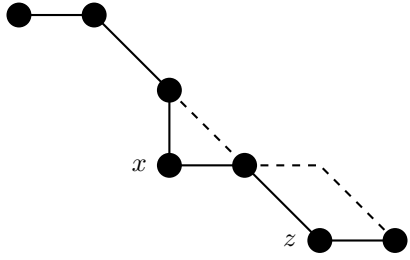

\begin{lemma}
\label{lem:weird}
Let $\Lambda$ be a  terrace and $x\in H_{\Lambda}$.  Then either $x$ is a corner of $\Lambda$ or there exists $i\in [d]$ such that $x-e_{i}\in H_{\Lambda}$.
\end{lemma}
\begin{proof}
If $x$ is not a corner of $\Lambda$ then there exists $i \in [d]$ such that $x-e_{i}\in \Lambda$.  Let $j \in [d]$.  Then $x+e_j\in \Lambda$ since $x\in H_{\Lambda}$.  Also $x-e_i+e_j\in \Lambda_+$ since $x-e_{i}\in \Lambda$.  These two facts together with Lemma \ref{lem:terrace}(vi) show that $x-e_i+e_j\in \Lambda$.  Since this holds for every $j\in [d]$ it follows that $x-e_i\in H_\Lambda$.
\end{proof}

\begin{lemma}
\label{lem:irreducible}
If $\Lambda$ is a terrace then:
\begin{enumerate}
\item[\normalfont(i)] if $x\in H_\Lambda$ and $x$ is a corner of $\Lambda$ then $\Lambda_{(\uparrow x)}= \Lambda\setminus\{x\}$;
\item[\normalfont(ii)] there exists an irreducible terrace $\Lambda'\subset \Lambda$.
\end{enumerate}
\end{lemma}
\begin{proof}
From Lemma \ref{pushingup} we know that $\Lambda\setminus\{x\}\subset \Lambda_{(\uparrow x)}\subset (\Lambda\setminus\{x\}) \cup \{x+e_j:j \in [d]\}$. But the fact that $x\in H_\Lambda$ means that each $x+e_j$ already $\in\Lambda\setminus\{x\}$. Therefore (i) holds.

To prove (ii), let $\Lambda^{\sss(0)}:=\Lambda$.  If $H_{\Lambda^{\sss(0)}}$ is empty then $\Lambda^{\sss(0)}=\Lambda$ is irreducible and we are done.  Otherwise we may enumerate the vertices of $H_{\Lambda^{\sss(0)}}$ as $(x^{\sss(m)})_{m \in \N}$.  

We proceed recursively.  Given $\Lambda^{\sss(n)}\subset \Lambda^{\sss(0)}$ with $H_{\Lambda^{\sss(n)}}$ nonempty, we let $M_n=\min\{m:x^{\sss(m)}\in H_{\Lambda^{\sss(n)}}\}$ (which is finite, since $H_{\Lambda^{\sss(n)}}\subset H_{\Lambda^{\sss(0)}}$). If $y^{\sss(n,0)}:=x^{\sss(M_n)}$ is a corner of $\Lambda^{\sss(n)}$ then define $\Lambda^{\sss(n+1)}=\Lambda^{\sss(n)}\setminus \{x^{\sss(M_n)}\}$.  Otherwise, by Lemma \ref{lem:weird}, there exists $i_1\in [d]$ such that $y^{\sss(n,1)}:=x^{\sss(M_n)}-e_{i_1}\in H_{\Lambda^{\sss(n)}}$.  If $y^{\sss(n,1)}$ is a corner of $\Lambda^{\sss(n)}$ then define $\Lambda^{\sss(n+1)}=\Lambda^{\sss(n)}\setminus \{y^{\sss(n,1)}\}$.  Otherwise we continue to find points $y^{\sss(n,s)}$ for $s\in \N$ such that $y^{\sss(n,s)}\in H_{\Lambda^{\sss(n)}}$ and $y^{\sss(n,s)}=y^{\sss(n,s-1)}-e_{i_s}$ for some $i_s$.  Let $s_n=\inf\{s:y^{\sss(n,s)} \text{ is a corner of $\Lambda^{\sss(n)}$}\}$.  Then $s_n$ is finite by Lemma \ref{lem:terrace}(viii), because each $y^{\sss(n,s)}\in x^{\sss(M_n)}_-$. We define $\Lambda^{\sss(n+1)}=\Lambda^{\sss(n)}\setminus \{y^{\sss(n,s_n)}\}$.  By (i), $\Lambda^{\sss(n+1)}$ is a terrace.

Thus we have produced a decreasing sequence $(\Lambda^{\sss(n)})_{n \in \Z_+}$ of terraces with $\Lambda^{\sss(0)}=\Lambda$.  By Lemma \ref{lem:decreasing_terraces}, the intersection of these, $\Lambda'$ is a terrace.  We claim that $\Lambda'$ is irreducible, i.e.~that $H_{\Lambda'}=\varnothing$.  To see this, first note that the sets $H_{\Lambda^{\sss(n)}}$ are decreasing since $\Lambda^{\sss(n)}$ are, and similarly $H_{\Lambda'} \subset H_{\Lambda^{\sss(n)}}$ for each $n$ since $\Lambda' \subset \Lambda^{\sss(n)}$ for each $n$.  
To complete the proof it is sufficient to show that for each $m$, $x^{\sss(m)}\notin H_{\Lambda^{\sss(n)}}$ for all $n$ sufficiently large (since then $H_{\Lambda'}=\varnothing$).  To verify this, it is (without loss of generality) sufficient to show that $x^{\sss(1)}\notin H_{\Lambda^{\sss(n)}}$ for all $n$ sufficiently large.  But $x^{\sss(1)}_-\cap \Lambda_+$ is finite by Lemma \ref{lem:terrace}(viii).  It follows that 
$x^{\sss(1)}\notin H_{\Lambda^{\sss(n)}}$ for $n>K_1+1$, where $K_1$ is the cardinality of the set $x^{\sss(1)}_-\cap \Lambda_+$ (since by this time we have removed $ \{y^{\sss(1,s_1)}, y^{\sss(1,s_1-1)}, \dots, y^{\sss(1,0)}\}$).
\end{proof}

\subsection{Local terraces}
\label{sec:localterr}
In this section $Q$ denotes a box of the form $[a,b]$ for some $a,b\in \Z^d$ (recall \eqref{boxdef}).
Its {\it boundary} $\partial Q$ is the set of sites in $Q$ that are adjacent to sites in $\mathbb{Z}^d\setminus Q$.  Its {\em interior} is $Q^o:=Q\setminus \partial Q$.  

We require a local formulation for the notion of a terrace. 
\begin{definition}
\label{def:Qterrace}
We call $\qt\subset Q$ a {\it $Q$-terrace} if $\qt=\Lambda\cap Q$ for some terrace $\Lambda$. 
\end{definition}
Note that $\varnothing$ is by definition a $Q$-terrace for any $Q$, although it is not an interesting one.
 The first observation is elementary.

\begin{lemma}
\label{lem:plus_sets}
If $\Lambda$ is a terrace and $\qt:=\Lambda \cap Q\ne \varnothing$ then $\Lambda_+\cap Q=\qt_+\cap Q$.
\end{lemma}
\begin{proof} Let $Q=[a,b]$. 
If $a \in \Lambda_+$ then Lemma \ref{lem:terrace}(i) and the fact that $\Lambda \cap Q$ is non-empty implies that $a\in \Lambda\cap Q=\qt$ and in this case the claim reduces to $Q=Q$.  Otherwise $a\notin \Lambda_+$.  Let $y\in \Lambda_+\cap Q$.  Find a nearest neighbour path, none of whose coordinates increases, from $y$ to $a$.  This path is necessarily in $Q$.  By Lemma \ref{lem:terrace}(vii) this path must hit some point $x\in \Lambda \cap Q=\qt$.  Therefore $y\in x_+\subset \qt_+$.  This proves that  $\Lambda_+\cap Q\subset \qt_+\cap Q$ and the reverse inclusion is trivial.
\end{proof}

A number of results that we proved for terraces immediately follow for $Q$-terraces, leading to the following result.
\begin{lemma}
\label{lem:QQresults}
Fix a box $Q$.
\begin{itemize}
\item[\normalfont(i)] If $\qt^{\sss(0)}\supset \qt^{\sss(1)}\supset \dots$ is a decreasing sequence of $Q$-terraces, then $\hat{\qt}:=\cap_{s=0}^\infty \qt^{\sss(s)}$ is a $Q$-terrace.
\item[\normalfont(ii)]  Let $\qt$ be a $Q$-terrace, and suppose $x,y\in\qt$. There is a path $z_{\sss(0)}, \dots,z_{\sss(n)}$ of points in $\qt$ such that 
 \begin{itemize}
\item[\normalfont(a)] $z_{\sss(0)}=x$, $z_{\sss(n)}=y$;
\item[\normalfont(b)] the $\ell_1$ distance from $x$ to $z_{\sss(k)}$ is strictly increasing in $k$, and the $\ell_1$ distance from $z_{\sss(k)}$ to $y$ is strictly decreasing;
\item[\normalfont(c)] each increment $z_{\sss(j+1)}-z_{\sss(j)}$ has the form $e_i$, $-e_k$, or $e_i-e_k$ for some $i\neq k$; 
\item[\normalfont(d)] each coordinate $z_{\sss(j)}^{\sss[i]}$ is monotone in $j$.
\end{itemize}
\item[\normalfont(iii)] Let $\qt$ be a $Q$-terrace.  If $x,y\in\mathbb{Z}^d\setminus \qt_+$ or $x,y\in\qt_+\setminus\qt$ then there is a self avoiding nearest neighbour path connecting $x$ and $y$ which is disjoint from $\qt$.  If in addition $x,y \in Q$ then there exists such a  path in $Q$.
\end{itemize}
\end{lemma}

\begin{proof}
A decreasing sequence of $Q$-terraces is trivially a $Q$-terrace (compare with Lemma \ref{lem:decreasing_terraces}), because the sequence terminates. This proves (i). Property (ii) holds with the same proof as Lemma \ref{standardpath}, since the paths constructed there stay in $Q$ if they start and end in $Q$. The proof of Lemma \ref{lem:connected} carries over in the same way, showing (iii).
\end{proof}

A set $A\subset Q$ is said to be \emph{$Q$-solid above} if $A_+\cap Q=A$.
\begin{lemma}
\label{lem:Qsolid}
If $B$ is solid above then $A:=B\cap Q$ is $Q$-solid above.
\end{lemma}
\begin{proof}
If $B$ is empty then the statement is vacuously true.  Since $A_+\cap Q \supset A$ we need only show that $A_+\cap Q\subset A$.   Let $x\in A_+\cap Q$.  Then $x \in A_+$ so $x\in y_+$ for some $y\in A=B\cap Q$, so $x \in B_+=B$ and $x \in Q$ as required.  
\end{proof}
\begin{lemma}
\label{lem:QL}
Let $\qt$ be a $Q$-terrace.  Then $Q\cap \qt_+\setminus \qt$ is $Q$-solid above.
\end{lemma}
\begin{proof}
Let $\Lambda$ be a terrace such that $\qt=Q\cap \Lambda$.  Then $\Lambda_+\setminus \Lambda$ is solid above by Lemma \ref{lem:terrace}(vi), and therefore $Q\cap \Lambda_+\setminus \Lambda$ is $Q$-solid above by Lemma \ref{lem:Qsolid}.  Now use Lemma \ref{lem:plus_sets} to see that $Q\cap \Lambda_+\setminus \Lambda=Q\cap \qt_+\setminus \qt$.
\end{proof}

 If $A$ is $Q$-solid above, define 
$$
\lambda^{\sss Q}_A=\{x\in A: \text{$\exists i\in[d]$ such that $x-e_i\in Q\setminus A$}\}.
$$
The following is the local analogue of parts of Lemma \ref{lem:terrace}.
\begin{lemma} Let $Q=[a,b]$, and let $A$ be $Q$-solid above.
\label{lem:Qterraceproperties}
\begin{enumerate}
\item[\normalfont(i)] If $A\neq Q$ then $(\lambda^{\sss Q}_A)_+=A_+$.
\item[\normalfont(ii)] $\lambda^{\sss Q}_A$ is a $Q$-terrace.
\end{enumerate}
\end{lemma}

\begin{proof}
If $A=\varnothing$ then (i) is trivially true. So let $A\neq Q$ be nonempty.  Then $\lambda^{\sss Q}_A$ is nonempty. 

Let $x\in (\lambda^{\sss Q}_A)_+$. Then there is a $y\in\lambda^{\sss Q}_A$ such that $x\in y_+$. But $y\in A$ so $x\in A_+$ too. Conversely, let $x\in A_+$, so there is a $y\in A$ with $x\in y_+$. Find a nearest neighbour path, none of whose coordinates increases, from $y$ to $a$. This path is necessarily in $Q$.  Moreover, it must exit $A$ somewhere, since our hypotheses imply that $a\notin A$. So it must include a $z\in A$ such that the next step is to some $z-e_i\in Q\setminus A$. Therefore $z\in\lambda^{\sss Q}_A$ and $x\in y_+\subset z_+\subset (\lambda^{\sss Q}_A)_+$. This shows (i).

Let $B=\{x-e_i: x\in A, x-e_i\notin Q\}$. We claim that 
\begin{equation}
\label{eqn:outsideboundary}
B_+\cap Q\subset A.
\end{equation} 
To see this, let $y\in B_+$, so $y\in (x-e_i)_+$ for some $x\in A$ such that $x-e_i\notin Q$.  Therefore $x^{[i]}=a^{[i]}$. If $y^{[i]}=x^{[i]}-1=a^{[i]}-1$ then $y\notin Q$. Whereas if $y^{[i]}\ge x^{[i]}$ then $y\in x_+\subset A_+$. Therefore $B_+\cap Q\subset A_+\cap Q=A$.

Now let $\Pi=\{x: \sum x^{[j]}>\sum b^{[j]}\}$, so $C=\Pi\cup A_+\cup B_+$ is good using \eqref{solidabove_union} and \eqref{boundedbelow_union} and the fact that $\Pi$ is saturated.   Thus $\lambda_C$ is a terrace.  We claim that $\lambda^{\sss Q}_A=Q\cap \lambda_C$ (which is sufficient to verify (ii)).  By \eqref{eqn:outsideboundary}, $C\cap Q=A_+\cap Q=A$, where we have also used the fact that $A$ is $Q$-solid above by assumption. If $x\in\lambda^{\sss Q}_A$ then $x\in A_+\cap Q=A$ and  there is an $i$ with $x-e_i\in Q\setminus A$. Therefore $x-e_i\notin A_+$ (since $A_+\cap Q=A)$, and by \eqref{eqn:outsideboundary} $x-e_i\notin C$.  But $x\in C$ and $x-e_i\notin C$ so (e.g.~by Lemma \ref{lem:terrace}(ii),(vii)) $x\in \lambda_C$. Conversely, if $x\in\lambda_C\cap Q$ then $x\in A$ and (by Lemma \ref{lem:terrace}(i),(ii)) there is an $i$ with $x-e_i\notin C$. If $x-e_i\notin Q$ then $x-e_i\in B$ which is impossible since $x-e_i\notin C$. Therefore $x-e_i\in Q\setminus A$, so $x\in\lambda^{\sss Q}_A$, and (ii) is proved.
\end{proof}

Recall from Definition \ref{def:good} that a \emph{good} set is one that is saturated, solid above, and bounded below.  Since a subset of $Q$ is obviously bounded below and not saturated, the natural way to define a $Q$-good set is to simply say that it is $Q$-solid above.

According to Lemma \ref{lem:terrace}(ii) there is a 1-1 correspondence between terraces and good sets.  The following Lemma indicates that this is not the case for $Q$-terraces and $Q$-good sets, so we include it to clarify the nature of  $Q$-terraces (even though it is not required for the proof of our main results).
In particular for any $Q$-good set $A$ there are minimal and maximal $Q$-terraces $\qt$ (namely $\lambda^{\sss Q}_A$ and $\lambda^{\sss Q}_A\cup (A \cap\partial Q)$) with $\qt_+\cap Q=A$. 

\begin{lemma} Let $Q=[a,b]$. Let $\qt$ be a $Q$-terrace.
\label{lem:QterracepropertiesB}
\begin{enumerate}
\item[\normalfont(i)] Set $\bar\qt=\lambda^{\sss Q}_{\qt_+\cap Q}$. Then $\bar\qt\subset \qt$, with $\qt\setminus\bar\qt\subset\partial Q$. 
\item[\normalfont(ii)] If $\qt$ is nonempty and $Q\not\subset\qt_+$ then $(\bar\qt)_+\cap Q=\qt_+\cap Q$.
\end{enumerate}
\end{lemma}
\begin{proof}
If $\qt=\varnothing$ then (i) and (ii) are vacuously true.  Let $\Lambda$ be a terrace such that $\Lambda\cap Q=\qt$. If $a\in\qt \subset \qt_+$ then $\bar\qt=\varnothing$, and the only nontrivial statement among (i) and (ii) is that $\qt\subset\partial Q$. But $a\in\Lambda$ so  $(a+\1)_+\subset \Lambda_+\setminus\Lambda$ by Lemma \ref{lem:terrace}(i).  Since  $Q^o\subset(a+\1)_+$ we have $Q^o\cap \Lambda=\varnothing$ so $\qt\subset\partial Q$.

We may therefore assume that $\qt$ is nonempty (which shows that $a\in \qt_-$) and $a\notin\qt$.  If $a\in \qt_+$ then $a\in \Lambda_+ \cap \Lambda_-$ so by Lemma   \ref{lem:Qterraceproperties}(iv) $a\in \Lambda$ which gives a contradiction.  Therefore $a \notin \qt_+$.

To show (i), let $x\in \bar\qt$. Then $x\in\qt_+\cap Q\subset \Lambda_+$ and there is an $i$ with $x-e_i\in Q\setminus\qt_+=Q\setminus \Lambda_+$ (the equality holds by Lemma \ref{lem:plus_sets}). In particular, $x\in\Lambda$ (Remark \ref{rem:terrace_neighbour})  so in fact $x\in\qt$, showing that $\bar\qt\subset\qt$. 

If $x\in\qt\setminus\bar\qt$ then $x\in\qt_+\cap Q$, but for every $j\in [d]$, $x-e_j\notin Q\setminus (\qt_+)$. Also $x\in\Lambda$ so there is an $i$ with $x-e_i\in \Z^d\setminus (\Lambda_+)\subset \Z^d\setminus (\qt_+)
$. 
Thus, $x-e_i\in \Z^d \setminus Q$, so $x\in \partial Q$.  This proves (i).

To prove (ii), note that by Lemma \ref{lem:Qterraceproperties}(i), $(\bar\qt)_+=(\qt_+\cap Q)_+$.  Since $\qt\subset  \qt_+\cap Q$ we have $\qt_+\subset (\qt_+\cap Q)_+$.  Similarly since $\qt_+\cap Q\subset \qt_+$ we have that  $(\qt_+\cap Q)_+\subset (\qt_+)_+=\qt_+$.  Thus $\qt_+=(\qt_+\cap Q)_+=(\bar\qt)_+$.
\end{proof}

\begin{remark}
\label{rem:lambdacubed} 
If $\rt'=\lambda^{\sss Q}_{\rt_+\cap Q}$, where $\rt$ is non-empty and $Q\not\subset \rt_+$ then by Lemma \ref{lem:QterracepropertiesB}(ii) we have that 
\begin{equation}
\rt'=\lambda^Q_{(\rt')_+\cap Q}.
\end{equation} 
It therefore follows from the definition of $\lambda^{\sss Q}$ that if $y\in (\rt')_+$ satisfies $y-e_i\in Q\setminus (\rt')_+$ for some $i$, then $y \in \rt'$.
\end{remark}

\begin{lemma}
\label{lem:Qsolidlower}
If $G$ is $Q$-solid above and $z$ is a corner of $G$ then $G\setminus \{z\}$ is $Q$-solid above.
\end{lemma}
\begin{proof}
Let $x \in G\setminus \{z\}$.  Then $z\notin x_+$ since $z$ is a corner of $G$ (so for each $i \in [d]$, $z\notin (x+e_i)_+$ and also $z\ne x$).  Let $y \in Q\cap x_+$.  If $y=x$ then $y \in G\setminus \{z\}$ and we are done.  Otherwise $y\in (x+e_j)_+$ for some $j$ with $x+e_j\in Q$.  Now $x,y\in Q$ and $y \in x_+$ implies that $[x,y]\subset Q$.  Since $y-e_j\in x_+$  and $y-e_j \in y_-$ we have that $y-e_j\in [x,y]$.  Thus $y-e_j\in x_+\cap Q\subset G$.  But  $y-e_j\in x_+$ also implies that $y-e_j\ne z$ so $y-e_j\in x_+\cap Q\setminus \{z\}$ as required.
\end{proof}
If $z$ is a corner of a $Q$-terrace $\qt$ we may define 
\begin{equation}
\qt_{(\uparrow z)}^{\sss Q}=\lambda^{\sss Q}_{Q\cap \qt_+\setminus\{z\}}.\label{Qpushupdef}
\end{equation}
This is a $Q$-terrace by Lemmas \ref{lem:Qsolid}, \ref{lem:Qterraceproperties}(ii) and \ref{lem:Qsolidlower}.  We call this \emph{the $Q$-terrace resulting from pushing up $z$}. If $\qt$ is a $Q$-terrace, we let 
\begin{equation}
H^{\sss Q}_\qt=\{x\in\qt: \text{for every $i\in[d]$, either $x+e_i\in\qt$ or $x+e_i\notin Q$}\}.\label{HQ}
\end{equation}
In particular, if $\Lambda$ is a terrace and $\qt=\Lambda\cap Q$ then $H_\qt^{\sss Q}\cap Q^o=H_{\Lambda}\cap Q^o$. We then have the following.
\begin{lemma}
\label{lem:Qterraceconsequences}
Let $Q=[a,b]$ and $\qt$ be a $Q$-terrace. 
\begin{enumerate}
\item[\normalfont(i)] If $z\in H_{\qt}^{\sss Q}$, then either $z$ is a corner of $\qt$ or there exists $i\in [d]$ such that $z-e_{i}\in H_{\qt}^{\sss Q}$;
\item[\normalfont(ii)] If $\qt=\lambda^{\sss Q}_G$ where $G$ is $Q$-solid above, and $z$ is a corner of $\qt$, then  
$\qt\setminus\{z\} \subset \qt^{\sss Q}_{(\uparrow z)}\subset Q\cap\big((\qt\setminus\{z\}) \cup \{z+e_j:j \in [d]\}\big)\subset Q\cap \qt_+$;
\item[\normalfont(iii)] In the setting of \emph{(ii)}, if in addition $z\in H_\qt^{\sss Q}$, then $\qt_{(\uparrow z)}^{\sss Q}= \qt\setminus\{z\}$.
\end{enumerate}
\end{lemma}
\begin{proof}
To prove (i), let $z\in H_{\qt}^{\sss Q}$, and suppose that $z$ is not corner of $\qt$.  Then there exists $i \in [d]$ such that $z-e_{i}\in \qt$.  Let $j \in [d]$.  Then $z+e_j\in \qt$ or $z+e_j\notin Q$ since $z\in H_{\qt}^{\sss Q}$.  Also $z-e_{i}+e_j\in \qt_+$ since $z-e_{i}\in \qt$.  We claim that either $z-e_{i}+e_j\in \qt$ or $z-e_{i}+e_j\notin Q$.  To see this note that if $z-e_i+e_j\in Q\cap \qt_+\setminus \qt$ then by Lemma \ref{lem:QL}, $z+e_j\in Q\cap \qt_+\setminus \qt$ which gives a contradiction.  Hence the claim holds.  Since this holds for every $j\in [d]$ it follows that $z-e_{i}\in H_{\qt}^{\sss Q}$, proving (i).

To prove (ii) and (iii), if $G\ni a$ then $G=Q$, so $\qt=\varnothing$ and both statements are vacuous.  Thus we may assume that $a\notin G$.  Note that the last inclusion in (ii) is trivial.

By Lemma \ref{lem:Qterraceproperties}(i) we have $(\lambda^{\sss Q}_G)_+=G_+$ so $\qt_+\cap Q=G_+\cap Q=G$.  Therefore $\qt^{\sss Q}_{(\uparrow z)}=\lambda^{\sss Q}_{G\setminus \{z\}}$. Thus the first inclusion in (ii) reduces to 
\begin{align}
\{x\in G\setminus \{z\} &\text{ and }\exists i \in [d] \text{ such that }x-e_i \in Q\setminus G\}\label{aargh}\\
&\subset \{x \in G\setminus \{z\} \text{ and }\exists i \in [d] \text{ such that }x-e_i \in (Q\setminus G)\cup \{z\}\}\nonumber
\end{align}
which is trivial.  For the second inclusion, suppose that $x$ is in the right hand side of \eqref{aargh}.  If $x-e_i=z$ for some $i$ then $x=z+e_i$ and we are done.  Otherwise $x-e_i\in Q\setminus G$, so $x\in \lambda_G^{\sss Q}=\qt$ (and $x\ne z$).  Since also $\lambda^{\sss Q}_G\subset Q$ (since $G\subset Q$), this proves (ii).

By definition of $H_\qt^{\sss Q}$,  if $z\in H_\qt^{\sss Q}$ and   $z+e_j\in Q$ then $z+e_j \in \qt\setminus \{z\}$, so $\{z+e_j:j\in [d]\}\cap Q\subset  \qt\setminus \{z\}$ in this case).  Thus (iii) follows from (ii).
\end{proof}
Now we will fix a box $R$ and consider various smaller boxes $Q\subset R$. We define the relative boundary $\partial_{\sss R}Q$ of $Q$ in $R$ to be the set of sites in $Q$ that are adjacent to sites in $R\setminus Q$ and the relative interior $Q^{\sss R:o}=Q\setminus \partial_{\sss R} Q$. 

\begin{remark}
\label{rem:pushup_subset}
It follows from Lemma \ref{lem:Qterraceconsequences}(ii) that $(\qt^Q_{(\uparrow z)})_+\subset \qt_+$ since the only elements in $\qt^Q_{(\uparrow z)}$ that may not be in $\qt$ are in $z_+\subset \qt_+$.
\end{remark}

\begin{lemma}
\label{specialpath}
Let $Q\subset R$ be boxes,  and let $\rt$ be an $R$-terrace. Let $x\in\rt\cap Q$ and suppose $\rt$ has no corners in $Q^{\sss R:o}\cap x_-$.  Then there exists $y \in \rt\cap\partial_{\sss R} Q$ such that we may construct a path in $\rt\cap Q$ from $x$ to $y$ as in Lemma \ref{lem:QQresults}(ii), but using only moves of the form $-e_i$. For each $i$, all the $-e_i$ moves occur consecutively. 
\end{lemma}
\begin{proof}
If $x\in\partial_{\sss R} Q$ then $y=x$ works. So assume $x\notin\partial_{\sss R} Q$. Therefore $x\in\rt\cap Q^{\sss R:o}$. 

Let $k_1=\inf\{k\ge 1:x-ke_1\notin \rt\cap Q^{\sss R:o}\}$ and $m_1=k_1-1$. Set $z^{\sss(k)}=x-ke_{1}\in\rt \cap Q^{\sss R:o}$ for $0\le k\le m_1$.     If $z^{\sss(m_1)}-e_1\in \partial_{\sss R} Q$ then we take $y=z^{\sss(m_1)}-e_1$ and our path is complete (with $k_1$ steps).  

Otherwise we proceed recursively.  For  $j<d$, given $k_1,\dots, k_j$, let $m_j=\sum_{i=1}^j (k_i-1)$, and assume that $z^{\sss(k)}\in\rt \cap Q^{\sss R:o}$ for $0\le k\le m_j$ gives a path using only moves $-e_1, \dots, -e_j$.  Let $k_{j+1}=\inf\{k\ge 1:z^{\sss(m_{j})}-ke_{j+1}\notin \rt\cap Q^{\sss R:o}\}$ and set $z^{\sss(m_j+k)}=z^{\sss(m_j)}-ke_{j+1}$ for $k\le k_{j+1}-1$. Take $m_{j+1}=m_j+k_{j+1}-1$. Note that it could happen that $k_{j+1}=1$, in which case $z^{\sss(m_{j+1})}=z^{\sss(m_j)}$.  

We may continue this process until one of two things happens. Either we find a $j$ with $z^{\sss(m_{j+1})}-e_{j+1}\in\partial_{\sss R} Q$, in which case we're done, with $1+m_{j+1}$ steps. Or we exhaust the possible $j\in [d]$ before we reach $\partial_{\sss R} Q$ and the construction terminates in a point $z^{\sss(m_d)}\in \rt\cap Q^{R:o}$. Assume the latter. We will see that this leads to a contradiction. 

We claim that for each $j$ either $z^{\sss(m_{j})}-e_{j}\notin R$ or $z^{\sss(m_{j})}-e_{j}\in R\setminus \rt_+$. To see this, assume $z^{\sss(m_{j})}-e_{j}\in R$.
The construction ensures that $z^{\sss(m_{j})}-e_{j}\notin \rt \cap Q^{\sss R:o}$, so either $z^{\sss(m_{j})}-e_{j}\notin Q^{\sss R:o}$ or $z^{\sss(m_{j})}-e_{j}\notin \rt$. Because $z^{\sss(m_{j})}-e_{j}\in R$, the former is ruled out by the assumption that $z^{\sss(m_{j})}-e_{j}\notin\partial_{\sss R} Q$. Therefore $z^{\sss(m_{j})}-e_{j}\in R\setminus \rt_+$ by Lemma \ref{lem:QL}), as claimed.

We know that $z^{\sss(m_{d})}\in z^{\sss(m_{j})}_{-}$ for each $j$. It follows that for each $j$ either $z^{\sss(m_{d})}-e_{j}\notin R$ or $z^{\sss(m_{d})}-e_{j}\in R\setminus \rt_+$. In particular, $z^{\sss(m_{d})}$ is a corner of $\rt$ in $Q^{R:o}\cap x_-$. 
But that is impossible, so in fact the construction must reach $\partial_{\sss R} Q$. 
\end{proof}

Let $Q\subset R$ be boxes and let $G$ be $R$-solid above (and $\neq\varnothing$ or $R$). Let $\rt=\lambda^{\sss R}_G$ and $A\subset G$. If $\rt$ has a corner $z_1$ in $Q^{\sss R:o}\setminus A$, then push it up to give a new $R$-terrace $\rt_1=\rt^{\sss R}_{(\uparrow z_1)}$. If $\rt_1$ has a corner $z_2$ in $Q^{\sss R:o}\setminus A$, then push it up to give a new $R$-terrace $\rt_2$. Repeat this process until we obtain an $R$-terrace with no corners in $Q^{\sss R:o}\setminus A$ (which could in principle $=\varnothing$, e.g. if $Q=R$ and $A=\varnothing$).   We will use the notation $\rt^{\sss Q,A}$ for any $R$-terrace obtained in this way (in principle it need not be unique).  Since $\lambda^{\sss R}_G\subset G$ and (from Lemma \ref{lem:Qterraceconsequences}(ii) or Remark \ref{rem:pushup_subset}) pushing up $\lambda^{\sss R}_G$ keeps one in $(\lambda^{\sss R}_G)_+$ we have that
\begin{equation}
\label{pushed_obvious}
\rt^{\sss Q,A}\subset G.
\end{equation} 

\begin{lemma}
\label{lem:minimalterrace}
Let $Q\subset R$ be boxes and let $G$ be $R$-solid above (and $\neq\varnothing$ or $R$) and $A\subset G$. Let $\rt=\lambda^{\sss R}_G$ and $A\subset G$. Then
\begin{enumerate}
\item[\normalfont(i)] $\rt^{\sss Q,A}=\rt$ off $Q$ (i.e.~$\rt^{\sss Q,A}\cap (R\setminus Q)=\rt\cap (R\setminus Q)$).
\item[\normalfont(ii)] $R\cap (\rt^{\sss Q,A})_+$ is the minimal set which is $R$-solid above, agrees with $G$ off $Q^{\sss R:o}$, and contains $ A$.  \item[\normalfont(iii)] If $x\in\rt^{\sss Q,A}\cap Q$ then there is a $z\in G\cap (A\cup\partial_{\sss R} Q)$ with $x\in z_+$. 
\end{enumerate}
\end{lemma}
\begin{proof}
By Lemma \ref{lem:Qterraceconsequences} (ii),  pushing up corners in $Q^{\sss R:o}\setminus A$ will never alter $\rt$ outside $Q$. Therefore (i) holds.

Let $V$ be the minimal set which is $R$-solid above, agrees with $G$ off $Q^{\sss R:o}$, and is such that $V\supset  A$. $V$ is well defined, since intersections of sets which are $R$-solid above are also $R$-solid above (and $G$ is an instance of such a set). Let
$$
W=
\{x: x \in R\cap y_+ \text{ for some $y\in G\setminus Q$ or $y\in G\cap(A\cup\partial_{\sss R} Q)$}\}.
$$
We will show that 
\begin{align}
\label{VW}
W= V= R\cap (\rt^{\sss Q,A})_+.
\end{align}
This clearly implies (ii). 
  We shall see below that \eqref{VW} also implies (iii).    
  
First note that (by Lemma \ref{lem:Qsolid}) $R\cap (\rt^{\sss Q,A})_+$ is $R$-solid above.  Next, since $(\rt^{\sss Q,A})_+\subset G_+=(\lambda_G^{\sss R})_+$ (the former e.g.~since $\rt^{\sss Q,A}\subset G_+\cap R=G$ and the latter e.g.~by Lemma \ref{lem:Qterraceproperties}(i)) we have 
\begin{equation}
R\cap (\rt^{\sss Q,A})_+\cap (Q^{\sss R:o})^c\subset G \cap (Q^{\sss R:o})^c.\label{subsets1}
\end{equation}  
On the other hand, $G \cap (Q^{\sss R:o})^c\subset  G_+\cap R=  (\lambda_G^{\sss R})_+\cap R$, 
and since  pushing up corners in $Q^{\sss R:o}\setminus A$ will never excise points of $(Q^{\sss R:o})^c$ (e.g.~by Lemma \ref{lem:Qterraceconsequences}(ii)) this proves that 
\begin{equation}
R\cap (\rt^{\sss Q,A})_+\cap (Q^{\sss R:o})^c\supset G \cap (Q^{\sss R:o})^c.\label{subsets2}
\end{equation}
Combining \eqref{subsets1} and \eqref{subsets2} we see that  $R\cap (\rt^{\sss Q,A})_+$  agrees with $G$ off $Q^{\sss R:o}$.  If $x\in A\subset G\subset R$ then $x\in A\cap (\lambda^{\sss R}_G)_+$.  By repeated applications of Lemma \ref{lem:Qterraceconsequences}(ii), since $x\in A$ is never removed, we have that 
$x\in (\rt^{\sss Q,A})_+$.  Thus $R\cap (\rt^{\sss Q,A})_+$ contains $A$.

We have therefore shown that $R\cap (\rt^{\sss Q,A})_+$ is $R$-solid above, agrees with $G$ off $Q^{\sss R:o}$ and contains $A$.  Therefore $R\cap (\rt^{\sss Q,A})_+\supset V$.

If $U$ is $R$-solid above, agrees with $G$ off $Q^{\sss R:o}$, and contains $A$, then it contains any $z\in G\setminus Q$ and any $z\in G\cap(A\cup\partial_{\sss R} Q)$. Therefore it contains $W$ too. Since $U$ is otherwise arbitrary, we have that $V\supset W$.

To verify \eqref{VW} it remains to show that $W\supset R\cap (\rt^{\sss Q,A})_+$. So let $x\in R\cap (\rt^{\sss Q,A})_+\subset G$ 
where the inclusion was shown in \eqref{subsets1}. If $x\notin Q$ then (since $x\in x_+$) we immediately conclude that $x\in W$.  If $x\in A_+$ then $x\in W$ too, by definition.  
So assume that $x\in Q\setminus A_+$, so that $x_-\subset A^c$.

Write $Q=[a,b]$. Then $a\in R$, and if $a\in(\rt^{\sss Q,A})_+$ then (by \eqref{pushed_obvious}) $a\in G$, and $a\in\partial_{\sss R} Q$ (since otherwise $R$ would have to take the form $[a,b']$ for some $b'$, which would in turn imply that $G=R$, which is excluded by our hypotheses). Therefore $x\in W$ by definition, since  $x\in a_+$.

Therefore we may assume that $a\notin(\rt^{\sss Q,A})_+$. Take a nearest neighbour path from $x$ to $a$, none of whose coordinates increase. By (vii) of Lemma \ref{lem:terrace} and the fact that $x\in a_+$, this path includes some point  $y\in Q\cap \rt^{\sss Q,A}\cap x_-$. 
By construction, $y_-\subset x_-\subset A^c$, so there are no corners of $\rt^{\sss Q,A}$ in $Q^{\sss R:o}\cap y_-$. 
There is therefore a path as in Lemma \ref{specialpath} (beware of the different roles of the points $x$ and $y$ there) from $y$ to some $z\in \qt^{\sss Q,A}\cap \partial_{\sss R} Q\subset G\cap \partial_{\sss R} Q$. Since $x\in y_+\subset z_+$ it follows that $x\in W$, as required. This shows \eqref{VW} and hence (ii).

To see that \eqref{VW} implies (iii) as well, suppose $x\in\rt^{\sss Q,A}\cap Q$.  By \eqref{VW}, $x\in W$ so  $x\in z_+$ for some $z\in G\setminus Q$ or $z\in G\cap (A\cup \partial_{\sss R}Q)$
. Take a nearest neighbour path from $x$ to $z$, none of whose coordinates increase. Every point $y$ along it has $y\in z_+\cap x_-$,  so $y\in G$.   If $z \in G\cap(A\cup \partial_{\sss R} Q)$ then we are done.   
Otherwise $z\in G\setminus Q$, so the path passes through a point $z'\in G \cap \partial_{\sss R} Q\cap x_-$, which shows (iii).
\end{proof}

\subsection{Relation with $\mc{C}_o$}
\label{sec:terracesandCo}
Recall the definition of $\Omega_1$, and consider the model of Example \ref{exa:horthant}, where 
\begin{equation}
E=\mc{E}_+ \text{ and }F=\mc{E}.
\end{equation}  
Let $\mc{T}_o$ denote the (random) set of terraces $\Lambda$ such that $o \in \Lambda_+$ and $\Lambda \subset \Omega_1$.  Then we have the following, which is already essentially stated in Proposition \ref{prop:CoStructure}, but is now proved for completeness.
\begin{lemma}
\label{lem:terraceC}
For  Example \ref{exa:horthant}, if $p\in (p_c(d),1)$ then a.s.~$\mc{T}_o$ is non-empty and $\mathcal{C}_o=\cap_{\Lambda \in \mc{T}_o}\Lambda_+$.  
\end{lemma}
\begin{proof}
By Theorem \ref{thm:phasetransition} we have that $L_x^{(i)}$ is finite for every $x\in \Z^d$ and $i \in [d]$ and $\mc{C}_o=\big(\cup_x (x+L_x^{(i)})\big)_{+i}$  for each $i$.  Moreover, by definition of $L_x^{(i)}$ we have that $x+e_i+L^{(i)}_{x+e_i}=x+L^{(i)}_{x}$ for each $x,i$ so $\mc{C}_o=\big(\cup_{y:y\cdot e_i=0} (y+L_y^{(i)})\big)_{+i}$.  Since $e_i\in \mc{E}_+$ for each $i$ we have that $\mc{C}_o$ is solid above.  It is saturated since $L_x^{(i)}<\infty$ for each $x,i$ and it is bounded below since $L_y^{(i)}>-\infty$ for each $i$ and $y:y\cdot e_i=0$.  Therefore $\mc{C}_o$ is good.  The terrace $\lambda_{\mc{C}_o}\subset \mc{C}_o$ satisfies $o \in \mc{C}_o=(\lambda_{\mc{C}_o})_+$ by Lemma \ref{lem:terrace}(ii).  Moreover since each $x\in \lambda_{\mc{C}_o}\subset \mc{C}_o$ has a neighbour $x-e_i\notin (\lambda_{\mc{C}_o})_+=\mc{C}_o$ (by Lemma \ref{lem:terrace}(i)) we must have that $x\in \Omega_1$.  Thus $\lambda_{\mc{C}_o}\in \mc{T}_o$, and $\cap_{\Lambda \in \mc{T}_o}\Lambda_+\subset (\lambda_{\mc{C}_o})_+=\mc{C}_o$.

Now let $\Lambda \in \mc{T}_o$, so that $\Lambda \subset \Omega_1$, and suppose that  $x\in \mathcal{C}_o\setminus\Lambda_+$.   There is a (self-avoiding) path from $o$ to $x$ consistent with the environment.  By Lemma \ref{lem:terrace}(vii) this path exits $\Lambda_+$ by moving from a vertex $y_0$ in $\Lambda$ to a vertex $y_1=y-e_i\notin \Lambda_+$ for some $i$.  Thus $y_0 \notin \Omega_1$ which contradicts $\Lambda \subset \Omega_1$.  This proves that $\mathcal{C}_o\subset\Lambda_+$, so $\mathcal{C}_o\subset \cap_{\Lambda \in \mc{T}_o}\Lambda_+$.  
\end{proof}

\begin{proof}[Proof of Proposition \ref{prop:CoStructure}]
As remarked above, $\mc{C}_o$ is solid above since $e_i\in \mc{E}_+$ for each $i$. It is saturated for $p<1$ since in that case $L_x^{(i)}<\infty$ for each $x,i$. This proves (i). Lemma \ref{lem:terraceC} directly implies (ii). For $p<p_c$, Theorem \ref{thm:phasetransition}(I) shows that $\mc{C}_o=\Z^d$. But the final argument in the proof of Lemma \ref{lem:terraceC} then shows that $\mathcal{C}_o\subset\Lambda_+$ for every $\Lambda\in \mc{T}_o$, so in fact  $\mc{T}_o$ is empty. This proves (iii).
\end{proof}




We recall the definitions of two infinite self-avoiding paths in $\mc{C}_x$: 
Define the Wn path $\NWP{x}=(x^{\sss(n)})_{n \in \Z_+}$ from $x^{\sss(0)}=x\in \Z^2$ 
 by following the arrow $-e_1$ whenever possible, and otherwise $e_2$.  Similarly we  define the Se path $\SEP{x}:=(y^{\sss(n)})_{n \in \Z_+}$ from $x\in \Z^2$ by following the arrow $-e_2$ whenever possible, and otherwise $e_1$.  We use the concept of Wn and Se paths in copies of $\Z^2$ within $\Z^d$ to prove the following.
\begin{lemma}
\label{NWpaths}
Let $p\in (0,1)$ and $u,v\in\mathbb{Z}^d$ and $j \in [d]$. There is a $k\in\mathbb{Z}$ such that $v+ke_j\in\mathcal{C}_u$. If $u\in v_-$ then $k\le 0$ and $u\cdot e_j\ge (v+ke_j)\cdot e_j$.
\end{lemma}
Note that the existence of $k$ already follows from the fact that $\mc{C}_u$ is good. But the explicit construction in the proof provides us with additional detail.

\begin{proof}
Without loss of generality, $j=d$ and $u=o\in \Z^d$.  Follow the $\mathbb{Z}e_1\times\mathbb{Z}e_d$-Wn path or Se path from $o=(0,\dots, 0)$ to find $z^{\sss(1)}\in\mc{C}_o$ such that $z^{\sss(1)}=(v_1,0,\dots, 0,n_1)$ (use the Wn path (in which case $n_1\ge 0$)  if $v_1<0$  and otherwise the Se path (in which case $n_1\le 0$) ).  Now follow the $\mathbb{Z}e_2\times\mathbb{Z}e_d$- Wn path or Se path from $z^{\sss(1)}$ to find $z^{\sss(2)}\in\mathcal{C}_o$ such that $z^{\sss(2)}=(v_1,v_2,0,\dots, 0,n_2)$.  Continue this procedure until the first $d-1$ coordinates match. If $o\in v_-$ then we are using Se paths each time, which means that the $z^{\sss(i)}_d$ never increase. 
\end{proof}

For $x\in Q$ we let $\mc{C}^{\sss Q}_x$ be the set of $y\in Q$ that may be reached from $x$ by following arrows from our random environment that begin and end in $Q$.   Since $\mc{E}_+\subset \mc{G}_y$ for every $y$ it is trivial that $\mc{C}^{\sss Q}_x$ is $Q$-solid above.  Moreover, it is easily seen (c.f.~the proof of Lemma \ref{lem:terraceC}) that $Q$-terraces $\qt\subset \Omega_1$ block paths that remain in $Q$, as follows.
\begin{lemma}
\label{lem:Qterrace_crossing}
If $\qt\subset \Omega_1$ is a $Q$-terrace, with $x\in \qt_+\cap Q$ and $y\notin \qt _+\cap Q$, then $y\notin \mc{C}^{\sss Q}_x$. 
\end{lemma}

We will need the following converse to this statement, which is a local version of Proposition \ref{prop:CoStructure}.
\begin{lemma}
\label{lem:localCoStructure}
Let $x\in Q=[a,b]$. Either $Q=\mc{C}^{\sss Q}_x$ or $\qt:=\lambda^{\sss Q}_{\mc{C}^{\sss Q}_x}$ is a $Q$-terrace satisfying $\qt\subset \mc{C}^{\sss Q}_x\cap \Omega_1$, such that $\qt_+=(\mc{C}^{\sss Q}_x)_+$.
\end{lemma}
\begin{proof}
Since $\mc{C}^{\sss Q}_x$ is $Q$-solid above, if $a\in \mc{C}^{\sss Q}_x$ then $Q=\mc{C}^{\sss Q}_x$.  Otherwise $a \notin \mc{C}^{\sss Q}_x$ so $\qt:=\lambda^{\sss Q}_{\mc{C}^{\sss Q}_x}$ is a $Q$-terrace (Lemma \ref{lem:Qterraceproperties}(ii)), and for each $y \in \qt$ there exists $i \in [d]$ such that $y-e_i \in Q\setminus \mc{C}^{\sss Q}_x$, so $y \in \Omega_1$.  By Lemma \ref{lem:Qterraceproperties}, $\qt_+=(\mc{C}^{\sss Q}_x)_+$ as claimed.
\end{proof}

\begin{remark}
It is worthwhile noting that in general for $p\in (p_c,1)$ either: $\lambda_{\mc{C}_o}$ is irreducible, or; $o \in \lambda_{\mc{C}_o}$ and $H_{\lambda_{\mc{C}_o}}\subset o_+$. The following proves this in the setting of $Q$-terraces
\end{remark}

\begin{lemma}
\label{lem:Qirreducible}
Let $x\in Q$ and set $\qt:=\lambda^{\sss Q}_{\mc{C}^{\sss Q}_x}$. Then either
\begin{enumerate}
\item[\normalfont(i)] $H_{\qt}^{\sss Q}=\varnothing$; or 
\item[\normalfont(ii)] $x\in\qt$ and $H_{\qt}^{\sss Q}\subset x_+$; or 
\item[\normalfont(iii)] $Q=\mc{C}^{\sss Q}_x$.
\end{enumerate}
\end{lemma}
\begin{proof} Assume that neither (i) nor (iii) is the case, so there is a $z\in H_{\qt}^{\sss Q}\subset \qt$. Then for each $i$, either $z+e_i\notin Q$ or $z+e_i\in\qt$, and in the latter case, Lemma \ref{lem:localCoStructure} implies that $z+e_i$ has environment $\mc{E}_+$. In particular, we cannot reach $z$ in one step from any site $z+e_i\in Q$. Since $z\in \mc{C}^{\sss Q}_x$ we conclude that either $z=x$ or some site $y_1=z-e_{j_1}\in\mc{C}^{\sss Q}_x$. 

In the latter case for each $i$, either $y_1+e_i\notin Q$, or $y_1+e_i\in Q\cap \qt_+$ (as $y_1\in \mc{C}^{\sss Q}_x\subset\qt_+$).  If $y_1+e_i\in Q\cap \qt_+$ then either: $i=j_1$ and $y_1+e_i=z$  or; $i\ne j_1$ and since $Q$ is a rectangle, $z+e_i\in Q$, so 
 $z+e_i\in \qt$.  Either way we conclude that $y_1+e_i\in \qt$ (in the case $i\ne j_1$ we have $y_1+e_i=(z+e_i)-e_j\in \qt_+$ and $z+e_i\in \qt$ so by Lemma \eqref{lem:QL} $y_1+e_i\in \qt$). Therefore $y_1\in H_{\qt}^{\sss Q}$.

We may repeat this argument (either $y_1=x$ or some site $y_2=y_1-e_{j_2}\in \mc{C}^{\sss Q}_x$ and so on), but eventually the iteration must stop because $Q$ is finite.  It stops precisely when $y_n=x$. In other words, there is a sequence $z=:y_0,\dots,y_n=x\in\qt$ with each $y_{k+1}\in (y_k)_-$.
Therefore $z\in x_+$ and (ii) holds.
\end{proof}

An event $A$ that is measurable with respect to $\sigma\big((\mc{G}_x)_{x\in \Z^d}\big)$ is said to be {\em increasing} if the following holds:
\begin{equation}
\text{for any }\Gamma_1\subset \Gamma_2\subset \Z^d \text{ if }\{\Omega_1=\Gamma_1\}\subset A \text{ then }\{\Omega_1=\Gamma_2\}\subset A.
\end{equation}
The following is a trivial consequence of the fact that $\mc{E}_+\subset \mc{E}$.
\begin{corollary}
\label{lem:increasingevent}
Let $x,y\in Q=[a,b]$. The event $\{\text{$\mc{C}^{\sss Q}_x$ doesn't intersect $y_-$}\}$ is increasing.
\end{corollary}

Let $x,y\in R\subset \Z^d$.  Let $\bso\in \{\mc{E}_+,\mc{E}\}^{\Z^d}$.  Write $\mc{C}_x({\bs\omega})$ for the set of vertices in $\Z^d$ reachable from $x$ by following arrows in ${\bs \omega}$ and similarly $\mc{C}_x^{\sss R}({\bs\omega})$ is the set of sites reachable using only arrows in $\bs\omega$ that begin and end in $R$.  
A site $u\in R$ is said to be {\it pivotal} for $x \to y$ in $\bs{\omega}_R:=(\bs{\omega}_z)_{z\in R}$ if there is no connection from $x$ to $y_-$ within $R$ if $\bs{\omega}_u$ is replaced by $\mc{E}_+$, but such a connection is created if we replace $\bs{\omega}_u$ by $\mc{E}$.  We write $\Omega_1(\bs{\omega}_R)=\{x\in R:\bs{\omega}_x=\mc{E}_+\}$, and $\Omega_1(\bs{\omega}):=\Omega_1(\bs{\omega}_{\Z^d})$.

\begin{lemma}
\label{lem:pivotalsitesandcorners} Let ${\bs \omega}\in \{\mc{E}_+,\mc{E}\}^{\Z^d}$ and let $R\ni x,y$ be a box, and suppose that there exists an $R$-terrace $\rt\subset \Omega_1(\bs{\omega}_R)$ such that $x\in \rt_+$ and $y \notin \rt_+$.
Let $u\in R\setminus x_+$.  If $u$ is pivotal for $x\to y$ in $\bs{\omega}_R$, then $u\in\rt$ and there is an $i$ with $u+e_i\in R\setminus \rt$, and a $j$ with $u-e_j\in R\setminus \rt_+$.
\end{lemma}
\begin{proof}
Let $u$ be pivotal for $x\to y$ in $\bs{\omega}_R$.  Then $u \in \mc{C}^{\sss R}_x(\bso)$ 
and $y \in \mc{C}^{\sss R}_{u-e_{i_0}}(\bso)$ for some $i_0\in [d]$.  The first observation plus Lemma \ref{lem:Qterrace_crossing} implies that $u \in \rt_+$.   The second observation implies that $u-e_{i_0} \in R\setminus \rt_+$ (since if $u \in \Delta_+\cap R$ and $y\notin \Delta_+\cap R$ then by Lemma  \ref{lem:Qterrace_crossing} $y \notin \mc{C}^{\sss R}_{u-e_{i_0}}(\bso)$). So the final conclusion of the Lemma holds with $j=i_0$. 

 Let $\Lambda$ be a terrace such that $\Lambda\cap R=\rt$.  By Lemma \ref{lem:plus_sets},  $u \in \Lambda_+$ and $u-e_{i_0} \in (\Lambda_+)^c$.  By Lemma \ref{lem:terrace}(vii) $u \in \Lambda$.  Hence $u \in \rt$.

Since $u \in \mc{C}^{\sss R}_x(\bso)$ there is a self-avoiding nearest-neighbour path from $x$ to $u$ that is consistent with the environment (in particular it is a subset of $ \mc{C}^{\sss R}_x(\bso)$) and stays in $R$.  By Lemma \ref{lem:Qterrace_crossing} in fact this path stays in $R\cap \rt_+$. Let $z\in \rt_+$ be the last point on this path from which the step taken is in $\mc{E}_-$  (there has to be at least one such point, otherwise we'd have $u\in x_+$).  In particular $z\notin \Omega_1(\bso)$ so $z\notin \rt$.  Therefore $z\in R\cap \rt_+\setminus\rt$, and there is an $i_1\in [d]$ such that $z-e_{i_1}\in R$ (being the next step in the path) but $u\in (z-e_{i_1})_+$.  

Since $u\in\rt$, we can't have $u\in z_+$, by Lemma \ref{lem:QL}. Therefore $u^{\sss[i_1]}=z^{\sss[i_1]}-1$ and so $u+e_{i_1}\in R$. Since $z\in \rt_+\setminus \rt$ and $u+e_{i_1}\in z_+$, Lemma \ref{lem:QL} implies that $u+e_{i_1}\in R\setminus\rt$. 
\end{proof} 

\section{Enhancement}
\label{sec:enhancement}

In this section, we will prove our main result.  Subsection \ref{disturbedorthant} will consider a disturbed version of the half-orthant model, and Subsection \ref{slabs} will then use this model to prove Theorem \ref{thm:p_c}.

Let $V_{d}\subset \Z^{d}$ denote the set of points $x=(x^{\sss[1]},\dots, x^{\sss[d]})$ such that 
\[h(x):=\sum_{i=1}^{\lceil d/2\rceil }i x^{\sss[2i-1]} - \sum_{i=1}^{\lfloor d/2\rfloor }i x^{\sss[2i]}\equiv 0 \mod \rd,\]
where $\rd$ is the smallest prime greater than $2{\lceil d/2\rceil }$.  When $d=2$ we have that $\rd=3$ and $h(x)=x^{\sss[1]}-x^{\sss[2]}$ and a finite portion of $V_2$ can be seen in Figure \ref{fig:V2}.
\begin{figure}
\begin{center}
\includegraphics[scale=.35]{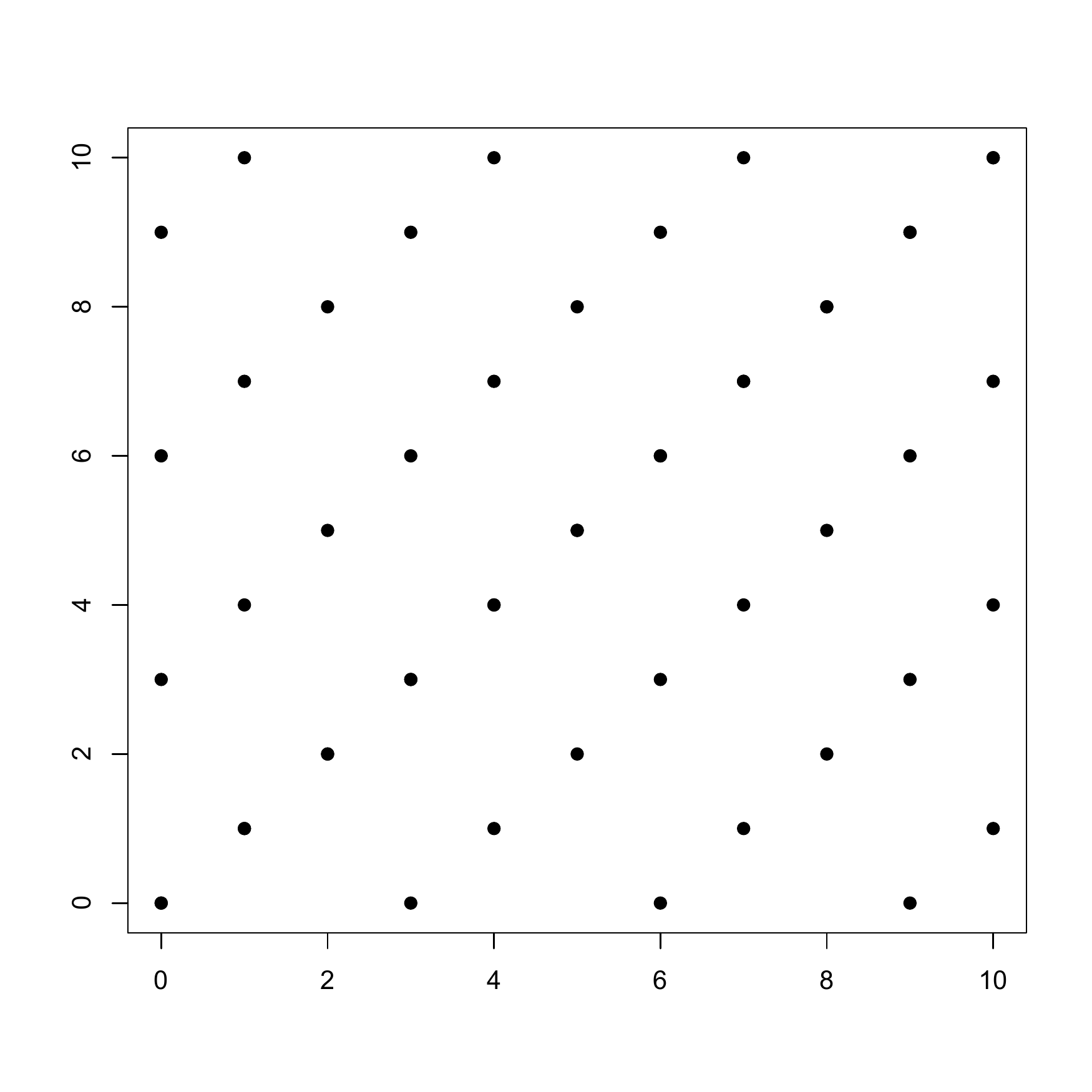}
\end{center}
\vspace{-1cm}
\caption{A finite portion of the set of points $V_2$.}
\label{fig:V2}
\end{figure}
Using the facts that $h(o)=0 \mod \rd$, and   $h(x+\rd e_{2i})=h(x)-\rd i$, and $h(x+\rd e_{2i-1})=h(x)+\rd i$,
 the following is easy to verify.
\begin{lemma}
\label{lem:Vd}
Any adjacent collection of at least $\rd$ vertices on a line parallel to one of the axes will contain at least one point of $V_d$.
\end{lemma}

Recall that the (orthant and) half-orthant model on $\Z^d$ was defined using uniform random variables $(U_x)_{x \in \Z^d}$. Given $p,q\in [0,1]$, 
\begin{itemize}
\item[\normalfont(i)] if $x\in \Z^d\setminus V_d$ then set $\mc{G}_{x}=\mc{E}_+$ if $U_x\le p$, and $\mc{G}_{x}=\mc{E}$ otherwise;
\item[\normalfont(ii)] if $x\in V_d$ then set $\mc{G}_{x}=\mc{E}_+$ if $U_x\le q$, and $\mc{G}_{x}=\mc{E}$ otherwise. 
\end{itemize}
Let $\mc{C}_x(p,q)\subset \Z^d$ be the set of sites that can be reached from $x$ by following the arrows of this model. Let $\theta(p,q)=\P\big(\mc{C}_o(p,q)=\Z^d\big)$.  As in the introduction it is trivial that $\mc{C}_x(p,q)$ is monotone decreasing in both $p$ and $q$. By definition of $\mc{C}_x(p,q)$ and Theorem \ref{thm:phasetransition}, we have that 
\begin{equation}
\P(\mc{C}_o(p,p)=\Z^d)=\begin{cases}
1, & \text{ if } p<p_c(d),\\
0, &\text{ if }p>p_c(d).
\end{cases}
\label{cpp}
\end{equation}

For an increasing continuous function $f:[0,1]\ra [0,1]$ such that $f(0)=0$, $f(1)=1$, and $f(p)<p$ for $p\in(0,1)$ we define what we call an {\it $f$-disturbed orthant model}, as follows.  Define 
$$
p_c(f,d):=\sup\{p:\theta(p,f(p))>0\}.
$$
Since $f$ is monotone, this shows that $\mc{C}_x(p,f(p))$ is monotone decreasing in $p$ and therefore $\theta(p,f(p))>0$ for all $p<p_c(f,d)$.  We will prove the following in Section \ref{sec:modification}.

\begin{proposition}
For $f$ continuous and increasing with $f(0)=0$, $f(1)=1$, and $f(p)<p$ for $p\in (0,1)$  \[p_c(f,d)>p_c(d).\]
\label{prp:strictdisturbedinequality}
\end{proposition}

Fix $d\ge 2$.
For $i\in\{1,2\}$ set $S^{\{i\}}=\Z^{d}\times \{i\}$ and $S^{\{1,2\}}=\Z^{d}\times \{1,2\}$, so $S^{\{1\}}, S^\{2\}\subset S^{\{1,2\}}\subset \Z^{d+1}$. Let $\mc{C}_{x,+1}(p)$ denote the forward cluster of $x$ in the half-orthant model for $\Z^{d+1}$. 

For $x\in S^{\{i\}}$ define $\mc{C}_{x}^{\{i\}}(p)\subset S^{^{\{i\}}}$ 
as the set of points of $S^{\{i\}}$ that can be reached from $x$ using only those arrows from the half-orthant model in $\Z^{d+1}$ that begin and end in $S_{i}$. Likewise, for $x\in S^{\{1,2\}}$ define $\mc{C}_x^{\{1,2\}}(p)$ as the set of points of $S^{\{1,2\}}$ that can be reached from $x$ using only those arrows from the orthant model in $\Z^{d+1}$ that begin and end in $S^{\{1,2\}}$. Therefore for $x\in S^{\{i\}}$ and $i\in\{1,2\}$ we have $\mc{C}_x^{\{i\}}(p)\subset\mc{C}_x^{\{1,2\}}(p)\subset\mc{C}_{x,+1}(p)$. 

As before, for $A\subset\mathbb{Z}^{d+1}$ let $A_+=\cup_{z\in A}z_+\subset\mathbb{Z}^{d+1}$. Now let $x\in S^{\{i\}}$, where $i\in \{1,2\}$. Clearly $p_c(d+1)=\sup\{p:\P\big( \mc{C}_{x,+1}(p)=\Z^{d+1}\big)>0\}$ and $p_c(d)=\sup\{p:\P\big( \mc{C}_{x}^{\{i\}}(p)=S^{\{i\}}\big)>0\}$. We may also define
\begin{equation}
p_c^{\{1,2\}}(d):=\sup\big\{p:\P\big( \mc{C}_{x}^{\{1,2\}}(p)=S^{\{1,2\}}\big)>0\big\}.\label{pc2d}
\end{equation}
From the above and the fact that a.s.~for all $p<1$ a positive proportion of vertices in any half line in $\Z^{d+1}$ are of type $\mc{E}$ (i.e.~are in $\Omega_0$) one can easily deduce that
\begin{equation}
p_c(d+1)\ge p_c^{\{1,2\}}(d)\ge p_c(d).
\label{slabinequality}
\end{equation}
We will show the following in the final section.
\begin{proposition} For $f(p)=p(1-(1-p)^{d+1})$ we have 
$p_c^{\{1,2\}}(d)\ge  p_c(f,d)$.
\label{slabstrictinequality}
\end{proposition}

\begin{proof}[Proof of Theorem \ref{thm:p_c}]  
Since $f(p)=p(1-(1-p)^{d+1})$ is the product of increasing functions, it is increasing, and the claim follows immediately from \eqref{slabinequality} together with Propositions \ref{prp:strictdisturbedinequality} and \ref{slabstrictinequality}. 
\end{proof}

\subsection{Disturbed orthant model}
\label{disturbedorthant}

\begin{lemma}
\label{lem:01}
For any $p,q\in (0,1)$, the following hold:
\begin{itemize}
\item[\emph{(1)}] for any $x\in \Z^d$, 
the events $\{\mc{C}_x(p,q)=\Z^d\}$ and $\{\mc{C}_x(p,q)\text{ is not bounded below}\}$ are almost surely equal, and 
\item[\emph{(2)}] for any $x\in S^{\{1,2\}}$ the events $\{\mc{C}^{\{1,2\}}_x(p)=S^{\{1,2\}}\}$ and $\{\mc{C}^{\{1,2\}}_x(p)\text{ is not bounded below}\}$ (as a subset of $\Z^{d+1}$) are almost surely equal, and 
\item[\emph{(3)}] each of the events in \emph{(1)} and \emph{(2)} has probability 0 or 1.
\end{itemize}
\end{lemma}
\begin{proof}
The proofs of both (1) and (2) mimic that of \cite[Lemma 2.2]{phase} as we now sketch.  Suppose that for some $y,i$, $\{k \in \N:y-ke_i\in \mc{C}_x\}$ is  infinite.  
Since $\mc{C}_x$ is solid above, in fact $y-ke_i\in \mc{C}_x$ for every $k \in \N$.  
But a.s.~infinitely many of the points $y-ke_i$ for $k \in \N$ are in $\Omega_0$ so $y-e_j-ke_i\in \mc{C}_x$ for infinitely many $k\in \N$ and for all $j$ (with the constraint that $y-e_j\in S^{\{1,2\}}$ in the case (2)).  This shows that if $\mc{C}_x$ is not bounded below then every site is in $\mc{C}_x$, and (1) and (2) follow.
 
The statement of (3) now follows as in the proof of Theorem 1.11 of \cite{phase}, using invariance under translation by vectors $e_i$ for $i\neq d+1$ in the case of $\mc{C}^{\{1,2\}}_x(p)$, and invariance under translation by vectors in $V_d$ in the case of $\mc{C}_x(p,q)$.
\end{proof}
Note that we will not actually use part (3) of Lemma \ref{lem:01} 
above, but have included it for completeness.
\begin{lemma}
Let $x,x'\in V_d$ be distinct.  Then $x'-e_i\notin \{x\}\cup \{x-e_j: j \in [d]\}$.
\label{lem:Edspreadout}
\end{lemma} 
\begin{proof}
To see this note that $x'- e_i$ cannot equal $x$, since if it did then
\[0<|h(x')-h(x)|\le \lceil d/2\rceil <\rd.\]
Similarly, if $x'-e_i=x-e_j$ (and $i\ne j$ since $x$ and $x'$ are distinct) then $x\in V_d$ and $x'=x+e_i-e_j\in V_d$.  However 
\[0<|h(x')-h(x)|\le 2  \lceil d/2\rceil <\rd,\]
so $h(x')-h(x)\ne 0 \mod \rd$, giving a contradiction. 
\end{proof}

Recall that $\1=\sum_{i\in [d]} e_i$.  The following result is the crucial technical result that allows us to use the technology of {\it enhancements} due to Aizenman and Grimmett (see \cite{AG91, Gbook}).   Given $n\in \N$, let $Q_n$ be the box $ 
 [-n,n]^d\cap\mathbb{Z}^d$.  Using the notation in \eqref{boxdef} we have that $Q_n=[-n\1,n\1]$.

Fix $N>M>n>\rho_d+4$ and set $R=Q_N$ and $S=(-M\1)_-\cap R$.  Let $Q_{n,x,N}=(x+Q_n)\cap R$.  See Figure \ref{fig:squares1} for a depiction in 2 dimensions.  Let $\mc{P}_{M}(\bso_R)$ denote the set of vertices (in $R$) that are pivotal for $o \to -M\1$  in $\bso_R$ (or equivalently $o \to S$ in $\bso_R$). 
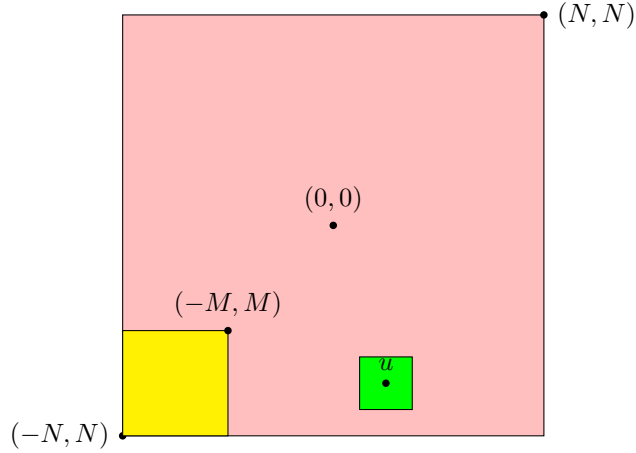
\begin{figure}
\begin{center}
\begin{tikzpicture}[scale=0.7]
\draw[fill=pink] (0,0)--(8,0)--(8,8)--(0,8)--(0,0);
\node[circle,fill=black,scale=0.3,label=left:{$(-N,N)$}] at (0,0)  {};
\draw[fill=yellow] (0,0)--(2,0)--(2,2)--(0,2)--(0,0);
\node[circle,fill=black,scale=0.3,label=above:{$(-M,M)$}] at (2,2)  {};
\node[circle,fill=black,scale=0.3,label=right:{$(N,N)$}] at (8,8)  {};
\node[scale=0.3,circle,fill=black,label={$(0,0)$}] at (4,4) {};
\draw[fill=green] (4.5,0.5)--(5.5,0.5)--(5.5,1.5)--(4.5,1.5)--(4.5,0.5);
\node[circle,fill=black,scale=0.3,label=above:{$u$}] at (5,1)  {};
\end{tikzpicture}
\end{center}
\caption{A depiction of $R=Q_N$, $S$, and $Q_{n,u,N}$ in 2 dimensions.}
\label{fig:squares1}
\end{figure}

\begin{proposition}
\label{prop:modification}
Let $n>\rd+4$, and $N>M>n$. Suppose that there exists $u\in \mc{P}_{M}(\bso_{R})\setminus V_d$.  
We may modify $\bso_{R}$ on $Q_{n,u,N}$, to give a new environment $\bar{\bso}_{R}$ for which $V_d\cap Q_{n,u,N}\cap \mc{P}_{M}(\bar{\bso}_{R})\ne \varnothing$.
\end{proposition}
Before we prove Proposition \ref{prop:modification} let us demonstrate how it is used to prove Proposition \ref{prp:strictdisturbedinequality}.  Let $|A|$ denote the cardinality of $A\subset \Z^d$.  For $\bso_{R}\in \{\mc{E}_+,\mc{E}\}^{R}$, write $P_{p,q}(\bso_{R})$ to denote $\P_{p,q}(\bsg_{R}=\bso_{R})$.

\begin{lemma}
\label{lem:pivotalinequality}
For every $p,q\in (0,1)$ and $n>\rd+4$ there exists
$\alpha(p,q,n)\in (0,\infty)$  
such that 
\begin{equation*}
\mathbb{E}_{p,q}\big[|\mc{P}_{M}(\bs{\mc{G}}_{R}) \setminus V_d|\big]\le\alpha(p,q,n) \mathbb{E}_{p,q}\big[|\mc{P}_{M}(\bs{\mc{G}}_{R}) \cap  V_d|\big]
\end{equation*}
for every $N>M>n$.  If $\pi=[r_1,r_2]$, where $0<r_1<r_2<1$, then for fixed $n$, $\sup\limits_{p,q\in \pi}\alpha(p,z,n)<\infty$.
\end{lemma}
\begin{proof}
Fix $n> \rd+4$ and assume $N>M>n$.  Define $R=Q_N$ and $Q_{n,u,N}$ as above. Suppose $u\in \mc{P}_{M}(\bs{\mc{G}}_{R})\setminus V_d$. Since $o\in V_d$, we know that $u\neq o$.  

If $\bso_{R}$ is a configuration for which $u\in \mc{P}_{M}(\bso_{R})\setminus V_d$  then 
by Proposition \ref{prop:modification} there exists a modification $\bar{\bso}_{R}$ (agreeing with $\bso_{R}$ on $R\setminus (Q_n+u)$) such that we can find a $\bar{u}\in \mc{P}_{M}(\bar{\bso}_{R})\cap V_d$. We can bound $\frac{P_{p,q}(\bso_{R})}{P_{p,q}(\bar{\bso}_{R})}$ by a constant $a_1=a_1(p,q,n)$, since this involves changing the state of at most $(2n+1)^d$ vertices.  

For any given $u$ we may likewise bound the number of $\bs{\omega}_{R}$ that can give rise to a given $\bs{\bar\omega}_{R}$ by a constant $a_2=a_2(n)$, since only the sites in $u+Q_n$ may change. Thus we may take $a_2=2^{(2n+1)^d}$ (though this is a vast overcount since most configurations counted won't make the given $u$ pivotal). For any given $\bar{u}\in R\cap V_d$ there are at most $a_3=a_3(n)=(2n+1)^d$ choices for $u$ that could make $u$ pivotal in this way, because $|u^{\sss[i]}-\bar u^{\sss[i]}|\le n$ for each $i$. Therefore 
\begin{align*}
\mathbb{E}_{p,q}\big[|\mc{P}_{M}(\bs{\mc{G}}_{R}) \setminus V_d|\big]&=\sum_{u\in R\setminus V_d}\,\,\sum_{\substack{\bs{\omega}_{R} \text{ with }\\u \text{ pivotal}}}P_{p,q}(\bs{\omega}_{R})\\
&\le \sum_{u\in R\setminus V_d}\,\,\sum_{
\substack{\bs{\omega}_{R} \text{ with }\\ u \text{ pivotal}}}\,\,
\sum_{\substack{\bar u \in R\cap V_d:\\ \|\bar u-u\|_\infty\le n}}\,\,
\sum_{\substack{\bs{\bar\omega}_{R} \text{ with }\bar u \text{ pivotal}:\\
\bs{\bar\omega}_{R} =\bs{\omega}\text{ off }Q_n+u }}P_{p,q}(\bs{\omega}_{R})\\
&\le \sum_{\bar u\in R\cap V_d}\,\, 
\sum_{\substack{\bs{\bar\omega}_{R} \text{ with }\\ \bar u \text{ pivotal}}}\sum_{\substack{u \in R\setminus V_d:\\ \| \bar u-u\|_\infty\le n}}\sum_{
\substack{\bs{\omega}_{R}\text{ with } u \text{ pivotal}:\\\bs{\omega}_{R} =\bs{\bar \omega}_{R} \text{ off }Q_n+u}}
P_{p,q}(\bs{\omega}_{R})\\
&\le a_1
\sum_{\bar u\in R\cap V_d}\,\, 
\sum_{\substack{\bs{\bar\omega}_{R} \text{ with }\\ \bar u \text{ pivotal}}}
P_{p,q}(\bs{\bar\omega}_{R})\sum_{\substack{u \in R\setminus V_d:\\ \|\bar u-u\|_\infty\le n}}
\sum_{
\substack{\bs{\omega}_{R}\text{ with } u \text{ pivotal}:\\ \bs{\omega}_{R} =\bs{\bar \omega}_{R} \text{ off }Q_n+u}}1\\
&\le a_1
a_2a_3\sum_{\bar u\in R\cap V_d}\,\, 
\sum_{\substack{\bs{\bar\omega}_{R} \text{ with }\\ \bar u \text{ pivotal}}}P_{p,q}(\bs{\bar\omega}_{R})\\
&=\alpha(p,q,n) \mathbb{E}_{p,q}[\text{$\#$ of pivotal vertices in $R\cap V_d$}]
\end{align*}
and the constants involved depend on $n$ but not $M$ or $N$. The boundedness statement is also clear from this argument.
\end{proof}

\begin{proof}[Proof of Proposition \ref{prp:strictdisturbedinequality} (assuming Proposition \ref{prop:modification})]
Fix $n>\rd+4$.
Let $N>M>n$ both be large and as before, work on a finite box $R=Q_N=[-N,N]^d\cap\Z^d$. Let $\beta(p,q)=1-\theta(p,q)$ be the probability that $\mc{C}_x(p,q)\ne \Z^d$. 

Let $\beta_M^N(p,q)$ be the probability that $o$ fails to connect to $S$ via any path lying in $R$. Then $\beta_M^N(p,q)$ decreases with $N$, increases with $M$, and 
\begin{equation}
\lim_{M\to\infty}\lim_{N\to\infty}\beta_M^N(p,q)=\beta(p,q).\label{double_limit}
\end{equation}
To see this, observe that $1-$ the limit equals the probability of the event $C$ that $o$ connects to every $(-M\1)_-$. Clearly $\P_{p,q}(C)\ge \P_{p,q}(\mc{C}_o(p,q)=\Z^d)=\theta(p,q)=1-\beta(p,q)$. Conversely for every $z\in\mathbb{Z}^d$ we can find an $M$ such that $z\in(-M\1)_+$. Therefore if $y\in(-M\1)_-\cap \mc{C}_o(p,q)$ we also have $z\in(-M\1)_+\subset y_+\subset \mc{C}_o(p,q)_+$. In other words, $\mc{C}_o(p,q)=\Z^d$ on the event $C$, showing that $\P_{p,q}(C)\le \P_{p,q}(\mc{C}_o(p,q)=\Z^d)$.  This verifies \eqref{double_limit}.

Fix $M$ and $N$. Consider the sites in $R$ that are pivotal for $o\to -M\1$ in $\mc{G}_{R}$ (as defined after Corollary \ref{lem:increasingevent}). By Corollary  \ref{lem:increasingevent}, the event that $o$ fails to connect to $S$ via a path lying in $R$ (under the coupling used above) increases with both $p$ and $q$. Therefore a version of Russo's formula applies (see \cite{Gbook}), showing that 
\begin{align*}
\frac{\partial}{\partial p}\beta_M^N(p,q)&=\mathbb{E}_{p,q}[\text{$\#$ of pivotal vertices in $R\setminus V_d$}]\\
\frac{\partial}{\partial q}\beta_M^N(p,q)&=\mathbb{E}_{p,q}[\text{$\#$ of pivotal vertices in $R\cap V_d$}].
\end{align*}
Note that Russo's formula requires that there be only finitely many sites under consideration, so this is the point where our proof requires $M$ and $N$ to be finite, i.e. the use of local terraces rather than terraces.

Proposition \ref{prop:modification} implies  Lemma \ref{lem:pivotalinequality}, from which Russo's formula gives that
$$
\frac{\partial}{\partial p}\beta_M^N(p,q)\le \alpha(p,q,n)
\frac{\partial}{\partial q}\beta_M^N(p,q).
$$

We now use an elementary argument (see \cite{Gbook}) to obtain an inequality of the form
$$
\beta_M^N(p,q)\ge \beta_M^N(p+\delta, q-a\delta).
$$

Choose $\epsilon_0>0$ so that $p_c+2\epsilon_0<1$ and $0<p_c-2\epsilon_0$. Set $\pi=[p_c-\epsilon_0,p_c+\epsilon_0]$, and let $a$ be the bound on $\alpha(p,q,n)$ over $\pi$, given by Lemma \ref{lem:pivotalinequality}. By continuity of $f$ and the fact that $p>f(p)$, we can find a $\delta_0>0$ such that $p-a\delta>f(p+\delta)$ whenever $p\in\pi$ and $0<\delta<\delta_0$. Reducing $\delta_0$ if necessary, we can ensure that $\delta_0(a+1)<\epsilon_0$.

 Let $0<\delta<\delta_0$. Take $p=q=p_c-\frac{\delta}{2}$. Then $p=q<p_c<p+\delta$, and all of $p,q,p+\delta,q-a\delta\in\pi$.   By the mean value theorem, there is some $\delta'$ between $0$ and $\delta$ such that 
\begin{align}
0\le &\beta_M^N(p+\delta,q-a\delta)\nonumber\\
&\qquad =\beta_M^N(p,q)+\delta
\Big[\frac{\partial \beta_M^N}{\partial p}(p+\delta',q-a\delta')-a
\frac{\partial \beta_M^N}{\partial q}(p+\delta',q-a\delta')\Big]\nonumber\\
&\qquad \le \beta_M^N(p,q)+\delta[\alpha(p,q,n)-a]
\frac{\partial \beta_M^N}{\partial q}(p+\delta',q-a\delta')\nonumber\\
&\qquad \le \beta_M^N(p,q).\label{betabeta}
\end{align}
By \eqref{cpp} and \eqref{double_limit}, $\lim_{M\to\infty} \lim_{N\to\infty}\beta_M^N(p,p)= \beta(p,p)=0$, so taking limits in \eqref{betabeta} we get that 
$\lim_{M\to\infty}\lim_{N\to\infty}\beta_M^N(p+\delta,p-a\delta)= 0$ too. Recalling that $p-a\delta>f(p+\delta)$, we conclude that $\beta(p+\delta, f(p+\delta))=0$. In particular, $p_c(f,d)\ge p+\delta>p_c(d)$, as required. 
\end{proof}

\subsection{Local modification}
\label{sec:modification}
In order to complete the proof of Proposition \ref{prp:strictdisturbedinequality} it therefore remains to prove Proposition \ref{prop:modification}.  Recall that $E=\mc{E}_+$ and $F=\mc{E}_+$.  Recall also that for boxes $Q\subset R$, $Q^{\sss R:o}$ denotes the set of sites in $Q$ that are not adjacent to sites in $R\setminus Q$.  

The strategy of the proof is to start with an $R$-terrace of sites in $\Omega_1$ that passes through $u$ (see e.g.~Figure \ref{fig:enhance0}), and then push up corners in a box around $u$, except for a few that are close to $u$. Because of those nearby sites where we don't push up, the resulting terrace must still pass close to $u$. We will show that the absence of corners means that either $u_+$ contains long straight segments of the terrace, or the geometry forces there to be long straight segments from near $u$ backwards to the boundary of the box. Either way, this terrace will contain a point $\bar u$ of $V_d$, and one can then modify $\bs{\omega}$ to make that point pivotal and the new terrace $\subset\Omega_1$. 

\begin{figure}
\begin{center}
\begin{tikzpicture}[scale=0.9]
\draw[fill=pink] (0,0)--(8,0)--(8,8)--(0,8)--(0,0);
\node[circle,fill=black,scale=0.3,label=left:{$(-N,-N)$}] at (0,0)  {};
\draw[fill=yellow] (0,0)--(2,0)--(2,2)--(0,2)--(0,0);
\node[circle,fill=black,scale=0.3,label=above:{$$}] at (2,2)  {};
\node[circle,fill=black,scale=0.3,label=right:{$(N,N)$}] at (8,8)  {};
\draw[fill=green] (3.5,0.5)--(6.5,0.5)--(6.5,3.5)--(3.5,3.5)--(3.5,0.5);
\node[circle,fill=black,scale=0.3,label=left:{$u$}] at (5,2)  {};
\node[circle,fill=black,scale=0.3,label=left:{$o$}] at (4,4)  {};

\draw[->,thick,-stealth] (6.5,-0.5)--(6.5,-0.25);
\draw[->,thick,-stealth] (6.5,-0.5)--(6.75,-0.5);

\draw[->,thick,-stealth] (6,0)--(6,0.25);
\draw[->,thick,-stealth] (6,0)--(6.25,0);
\draw[->,thick,-stealth] (6,0.5)--(6,0.75);
\draw[->,thick,-stealth] (6,0.5)--(6.25,0.5);
\draw[->,thick,-stealth] (5.5,1)--(5.75,1);
\draw[->,thick,-stealth] (5.5,1)--(5.5,1.25);

\draw[->,thick,-stealth] (5,1.5)--(5,1.75);
\draw[->,thick,-stealth] (5,1.5)--(5.25,1.5);

\draw[->,thick,-stealth] (5,2)--(5,2.25);
\draw[->,thick,-stealth] (5,2)--(5.25,2);
\draw[->,thick,-stealth] (5,2.5)--(5,2.75);
\draw[->,thick,-stealth] (5,2.5)--(5.25,2.5);

\draw[->,thick,-stealth] (4.5,3)--(4.75,3);
\draw[->,thick,-stealth] (4.5,3)--(4.5,3.25);

\draw[->,thick,-stealth] (4,3)--(4.25,3);
\draw[->,thick,-stealth] (4,3)--(4,3.25);

\draw[->,thick,-stealth] (3.5,3)--(3.75,3);
\draw[->,thick,-stealth] (3.5,3)--(3.5,3.25);
\draw[->,thick,-stealth] (3,3)--(3.25,3);
\draw[->,thick,-stealth] (3,3)--(3,3.25);

\draw[->,thick,-stealth] (2.5,3.5)--(2.75,3.5);
\draw[->,thick,-stealth] (2.5,3.5)--(2.5,3.75);

\draw[->,thick,-stealth] (2,4)--(2.25,4);
\draw[->,thick,-stealth] (2,4)--(2,4.25);
\draw[->,thick,-stealth] (1.5,4)--(1.75,4);
\draw[->,thick,-stealth] (1.5,4)--(1.5,4.25);

\draw[->,thick,-stealth] (1,4)--(1.25,4);
\draw[->,thick,-stealth] (1,4)--(1,4.25);

\draw[->,thick,-stealth] (0.5,4)--(0.75,4);
\draw[->,thick,-stealth] (0.5,4)--(0.5,4.25);

\draw[->,thick,-stealth] (0,4.5)--(0.25,4.5);
\draw[->,thick,-stealth] (0,4.5)--(0,4.75);

\draw[->,thick,-stealth] (0,5)--(0.25,5);
\draw[->,thick,-stealth] (0,5)--(0,5.25);

\draw[->,thick,-stealth] (0,5.5)--(0.25,5.5);
\draw[->,thick,-stealth] (0,5.5)--(0,5.75);

\draw[->,thick,-stealth] (-0.5,6)--(-0.25,6);
\draw[->,thick,-stealth] (-0.5,6)--(-0.5,6.25);

\end{tikzpicture}
\end{center}
\caption{A depiction in 2 dimensions of an $R$ terrace in $\Omega_1$, passing through $u$.}
\label{fig:enhance0}
\end{figure}

The argument is complicated by the possibility of $u$ lying close to a side of $R$, or close to $o_+$, or close to $S$. All these possibilities result in special cases of the argument. The reader is advised to ignore the caveats and special cases at first, i.e. on first reading to assume that $u+Q_n$ intersects none of $o_+$, $\partial R$, or $S$.

\begin{proof}[Proof of Proposition \ref{prop:modification}]
{\hphantom{blank}}

Recall that $N>M>n>\rd+4$.
Let $R=Q_N$, $Q:=Q_{n,u,N}=(u+Q_n)\cap R$ and $S=(-M\1)_-\cap R$.
Set  $\bar Q=R\cap (u+Q_{n-2})\subset Q$.   We note here that 
\begin{equation}
\bar Q^{\sss R:o}\subset \bar Q\subset Q\setminus \partial_{\sss R}Q=:Q^{\sss R:o},\label{bar_int}
\end{equation}
 since if $y\in \bar Q$ then $\|y-u\|_\infty\le n-2$, while all $y\in \partial_{\sss R} Q$ have $\|y-u\|_\infty=n$, and all $y\in \partial_{\sss R} \bar Q$ have $\|y-u\|_\infty=n-2$.

Let $\bso_R$ and $u\in \mc{P}_{M}(\bso_{R})\setminus V_d$ be as in the statement of the lemma.  
Let $\bso^1$ and $\bso^0$ be in $\{\mc{E}_+,\mc{E}\}^R$, with $\bso^1_u=\mc{E}_+$ and $\bso^0_u=\mc{E}$ and both agreeing with $\bso_R$ on $R\setminus \{u\}$.
\bigskip

\noindent 
{\bf Terrace properties:} 
Since $u$ is pivotal for $o \to -M\1$,  by Lemma \ref{lem:localCoStructure},  $\rt=\lambda^{\sss R}_{\mc{C}^{\sss R}_o(\bso^1)}$ is an $R$-terrace satisfying $\rt\subset \mc{C}^{\sss R}_o(\bso^1) \cap \Omega_1$, and $\rt_+=(\mc{C}^{\sss R}_o(\bso^1))_+$, and $S\subset R\setminus \rt_+$.  Since $o\in \rt_+$ and $-M\1\notin \rt_+$ (and $u$ is pivotal for $o \to -M\1$), by Lemma \ref{lem:pivotalsitesandcorners} $u\in \rt$ and there is an $i$ with $ u+e_i \in R\setminus \rt$.  Because $\rt=\lambda^{\sss R}_{\mc{C}^{\sss R}_o(\bso^1)}$,  Lemma \ref{lem:Qterraceproperties} implies that $R\cap \rt_+=R\cap \mc{C}^{\sss R}_o(\bso^1)=\mc{C}^{\sss R}_o(\bso^1)$ (recall that $\mc{C}^{\sss R}_o(\bso^1)$ is $R$-solid above).  Therefore 
\begin{equation}
\rt=\lambda^{\sss R}_{R\cap \rt_+}.  \label{Lambda=lambda}
\end{equation}
Moreover, from Lemma \ref{lem:Qirreducible} we have that
\begin{equation}
H^{\sss R}_{\qt}\subset o_+. \label{Hsuboplus}
\end{equation}

\bigskip

We will find an $R$-terrace $\OLa$ with the following properties:
\begin{align}
&\text{$\OLa=\lambda^{\sss R}_{R\cap(\OLa)_+}$}\label{eqn0}\\
&\text{$\OLa=\rt$ on $R\setminus Q^{\sss R:o}$, i.e.~$\OLa\cap (R\setminus Q^{\sss R:o})=\rt\cap (R\setminus Q^{\sss R:o})$
}\label{eqn1}\\
&\text{$(\OLa)_+=\rt_+$ on $R\setminus Q^{\sss R:o}$, i.e.~$(\OLa)_+\cap (R\setminus Q^{\sss R:o})=\rt_+ \cap (R\setminus Q^{\sss R:o})$}\label{eqn2}\\
&(\OLa)_+\subset\rt_+\label{eqn6}\\
& o\in(\OLa)_+\label{eqn3}\\
& S\subset R\setminus(\OLa)_+\label{eqn4}\\
&\text{$\exists \bar u\in\OLa\cap V_d\cap \bar{Q}$ such that either $\bar u\notin H^{\sss R}_{\OLa}$ or $\bar u\in o_+$.}\label{eqn5}
\end{align}
We'll assume for now that this can be done, and will show how to modify $\bs{\omega}$ on $Q$ so as to make $\bar u\in V_d\cap Q$ (appearing in \eqref{eqn5}) pivotal.
\bigskip

\noindent{\bf Finding paths in $R$:}
Define $\bs{\bar \bso}^1\in \{\mc{E}_+,\mc{E}\}^R$ as follows.   
Let $\bs{\bar\omega}^1$ agree with $\bs{\omega}^1$ (and hence $\bso_R$) off $Q$.  Choose $\bs{\bar\omega}^1$ to be $\mc{E}_+$ on $\OLa$, and note that this is consistent with $\bso^1$ off $Q$, by \eqref{eqn1}.

Because $\OLa$ is an $R$-terrace, Lemma \ref{lem:Qterrace_crossing} together with \eqref{eqn3} and \eqref{eqn4} shows that there is no path in $\bs{\bar\omega}^1$ from $o$ to $S$, regardless of how $\bs{\bar\omega}^1$ is defined at the remaining sites of $Q$. Set $\bs{\bar\omega}^1$ to be $\mc{E}$ at any such remaining sites of $Q$.

Now let $\bs{\bar\omega}^0$ agree with $\bs{\bar\omega}^1$ at all sites, except at $u$ where we take $\bs{\bar\omega}^0_u=\mc{E}$ (note that $\bs{\bar\omega}^1_u=\mc{E}_+$ since $u \in \OLa$).

Because $u$ is pivotal, we know that $\bs{\omega}^0$ connects $o$ to some $y^{\sss(0)}\in S$. Taking its initial and final portions, we have a (self-avoiding) path in $R$ consistent with $\bs{\omega}^0$ that runs from $o$ to $u$ and another from $u$ to $y^{\sss(0)}$.  That first segment of path lies in $\mc{C}^{\sss R}_o(\bso^1)\subset\rt_+$ since it does not make use of the environment at $u$, i.e.~is consistent with $\bso^1$ as well as $\bso^0$.

Let $y^{\sss(1)}$ be the first point in the former path that is in $Q$ (this is well defined since $u\in Q$), and let $y^{\sss(2)}$ be the last point that the latter path encounters $Q$.
In other words, there is a path from $o$ to $y^{\sss(1)}\in \rt_+$ 
that encounters $Q$ only at its last step, and a path from $y^{\sss(2)}$ to $y^{\sss(0)}$ that never returns to $Q$ after its first step. We claim these paths are consistent with $\bs{\bar\omega}^0$, and therefore that 
\begin{equation}
\text{$\bs{\bar\omega}^0$ connects $o$ to $y^{\sss(1)}$ and $y^{\sss(2)}$ to $y^{\sss(0)}$}. \label{claaaim}
\end{equation}
Note that $\bs{\omega}^0$ makes these connections, that $\bs{\bar\omega}^0=\bs{\omega}^0$ off $Q$, and that the only environment in $Q$ used by these paths is at $y^{\sss(2)}$. We will show that $y^{\sss(2)}\notin\OLa$, so $\bs{\bar\omega}^0_{y^{\sss(2)}}=\mc{E}$ by construction, which makes $\bs{\bar\omega}^0$ consistent with whatever step our path takes to leave $y^{\sss(2)}$.  This will therefore be sufficient to verify the claim.

If $Q$ intersects $S$ it could happen that $y^{\sss(2)}=y^{\sss(0)}\in S\cap Q$ (just as it could happen that $y^{\sss(1)}=o$, if $o\in Q$, though otherwise $y^{\sss(1)}\in\partial_{\sss R} Q$). Regardless, we know that either $y^{\sss(2)}\in S$ or $y^{\sss(2)}\in\partial_{\sss R} Q$. Either way, we have that $y^{\sss(2)}\neq u$. Any path in $R$ from $\rt_+$ to $S$, that is consistent with $\bs{\omega}^0$, must pass through $u\in Q$. Therefore $y^{\sss(2)}\notin\rt_+$ by definition,  so $y^{\sss(2)}\notin(\OLa)_+$ by \eqref{eqn6}.  In particular, $y^{\sss(2)}\notin\OLa$, establishing the claim.
\bigskip

We now show that $\bar u$ connects to $S$ in $\bs{\bar\omega}^0$ and $o$ connects to $\bar u$ in $\bs{\bar\omega}^0$.  This will show that $o$ connects to $S$ in $\bs{\bar\omega}^0$.  On the other hand $o$ does not connect to $ S$ in $\bs{\bar\omega}^1$ (which agrees with $\bs{\bar\omega}^0$ except at $\bar u$) showing that $\bar u$ is pivotal.

\bigskip

\noindent {\bf Connecting $\bar u$ to $S$:}
Because $\bar u\in\bar Q\cap\OLa$, \eqref{eqn0} allows us to find an $i \in [d]$ such that  $\bar u-e_i\in R\setminus (\OLa)_+$ (and $\bar u-e_i\in Q_{n-1}$). Therefore in fact $\bar u-e_i\in Q\setminus(\OLa)_+$. Using (iii) of Lemma \ref{lem:QQresults}, there is a self-avoiding nearest neighbour path in $Q\setminus(\OLa)_+$ that connects $\bar u-e_i$ to $y^{\sss(2)}$. All the $\bs{\bar\omega}^0$ environments along this path 
are of type $\mc{E}$, so 
$\bs{\bar\omega}^0$ connects $\bar u$ to $y^{\sss(2)}$ and hence to $S$.
\bigskip

\noindent {\bf Connecting $o$ to $\bar u$:}
\eqref{eqn5} tells us that either $\bar u\notin H^{\sss R}_{\OLa}$ or $\bar u\in o_+$.  In the latter case  
we may trivially connect $o$ to $\bar u$ in $\bs{\bar\omega}^0$ (since such a connection can be made using only $+$ steps, which are available from every site) as required.

We may therefore assume that $\bar u\notin H_{\OLa}^{\sss R}$. Accordingly, we may find $\bar u+e_j\in R\cap(\OLa)_+\setminus\OLa$, and since $\bar u \in \bar Q$, in fact $\bar u+e_j\in Q\cap(\OLa)_+\setminus\OLa$.  As remarked above, either $y^{\sss(1)}\in\partial_{\sss R} Q$ or $y^{\sss(1)}=o\in V_d$, so we must have $y^{\sss(1)}\neq u$.  Since $y^{\sss(1)}\in\rt_+$, using either \eqref{eqn2} (when $y^{\sss(1)}\in\partial_{\sss R} Q$) or \eqref{eqn3} we conclude that $y^{\sss(1)}\in(\OLa)_+$.
 
 If $y^{\sss(1)}\notin\OLa$ then (iii) of Lemma \ref{lem:QQresults} allows us to connect $y^{\sss(1)}$ to $\bar u+e_j$ via a self-avoiding path in $Q\cap(\OLa)_+\setminus\OLa$.  All the $\bs{\bar\omega}^0$ environments along this path will be of type $\mc{E}$, including that at $\bar u+e_j$. Therefore $\bs{\bar\omega}^0$ connects $y^{\sss(1)}$ to $\bar u$ and hence (by \eqref{claaaim}) also to $\bar u$.  This shows that $\bar u$ is pivotal when $y^{\sss(1)}\notin\OLa$.

It is however possible that $y^{\sss(1)}\in\OLa$, in which case we require an additional argument. In this case, 
we may take a step from $y^{\sss(1)}$ in some direction $e_k$ that keeps us in $Q$, provided such a $k$ exists. While we are still in $\OLa$, repeat. Eventually one of the following will happen: 
\begin{itemize}
\item[(i)] we reach $\bar u$ itself; or
\item[(ii)] we reach some point $y^{\sss(3)}\in Q\cap(\OLa)_+\setminus\OLa$; or 
\item[(iii)] we exhaust the possible $+$ steps that keep us in $Q$, before leaving $\OLa$ (i.e.~we reach the point  $b$, where $Q=[a,b]$, and $b\in\OLa$). 
\end{itemize}
In fact (iii) cannot happen, as we now argue. Recall that we had found a $\bar u+e_j\in Q\cap(\OLa)_+\setminus\OLa$.   But $ b\in (\bar u+e_j)_+$, so Lemma \ref{lem:QL} rules out having $ b\in\OLa$.

In the case (i) we're done, i.e.~$\bs{\bar\omega}^0$ connects $y^{\sss(1)}$ to $\bar u$. In the case (ii), since $y^{\sss(3)}\in Q\cap(\OLa)_+\setminus\OLa$,  Lemma \ref{lem:QQresults} allows us to connect $y^{\sss(3)}$ to $\bar u+e_j$ via a self-avoiding path in $Q\cap(\OLa)_+\setminus\OLa$ and (as in the case $y^{\sss(1)}\in Q\cap(\OLa)_+\setminus\OLa$) we conclude that $\bs{\bar\omega}^0$ connects $y^{\sss(1)}$ to $y^{\sss(3)}$ to $\bar u+e_j$ to $\bar u$.  

Thus we have shown that $\bar u$ is pivotal.

\bigskip

We now turn to the construction of $\OLa$ and $\bar u$ (see e.g.~Figure \ref{fig:constr}), broken into three cases.   \emph{ We assume for now that $o\notin \bar Q$. This assumption will be in place for the first two cases, and we will come back to the case $o\in \bar Q$ at the end of the proof. }

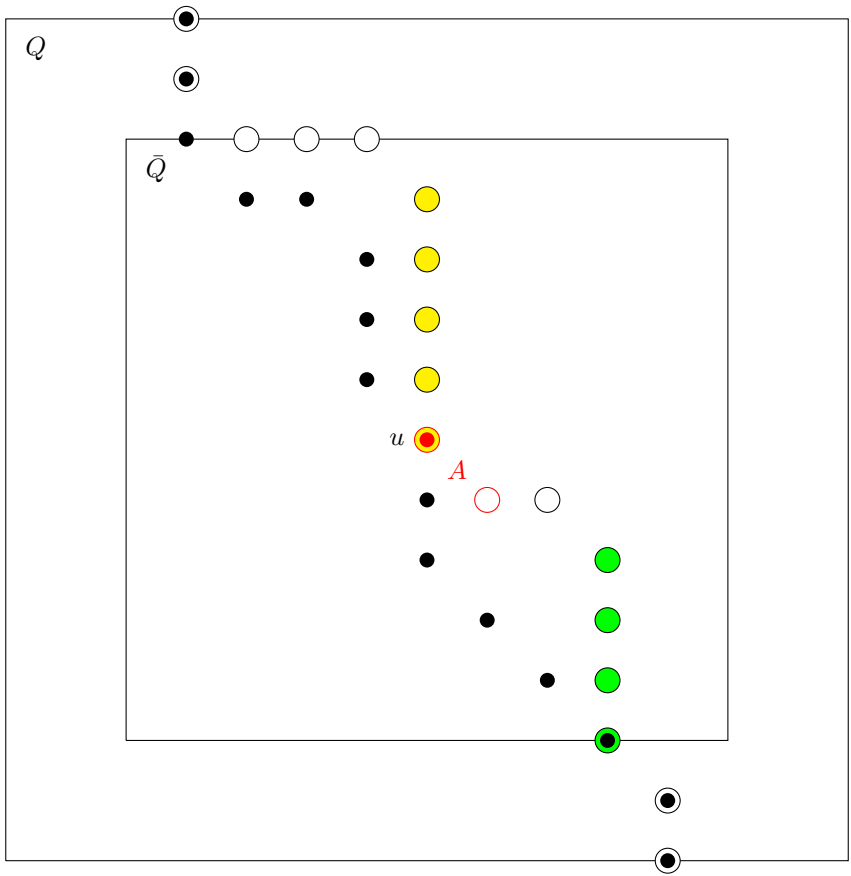
\begin{figure}[h]
\centering
\begin{tikzpicture}[scale=0.8]
\draw (0,0)--(14,0)--(14,14)--(0,14)--(0,0);
\draw (0.5,13.5) node {$Q$};
\draw (2,2)--(12,2)--(12,12)--(2,12)--(2,2);
\draw (2.5,11.5) node {$\bar Q$};
\draw (3,14) node[circle,draw=black,fill=white]  {};
\draw (3,13) node[circle,draw=black,fill=white]  {};
\draw (4,12) node[circle,draw=black,fill=white]  {};
\draw (5,12) node[circle,draw=black,fill=white]  {};
\draw (6,12) node[circle,draw=black,fill=white]  {};
\draw (7,11) node[circle,draw=black,fill=yellow]  {};
\draw (7,10) node[circle,draw=black,fill=yellow]  {};
\draw (7,9) node[circle,draw=black,fill=yellow]  {};
\draw (7,8) node[circle,draw=black,fill=yellow]  {};
\draw (7,7) node[circle,draw=red,fill=yellow]  {};
\draw (8,6) node[circle,draw=red,fill=white]  {};
\draw (9,6) node[circle,draw=black,fill=white]  {};
\draw (10,5) node[circle,draw=black,fill=green]  {};
\draw (10,4) node[circle,draw=black,fill=green]  {};
\draw (10,3) node[circle,draw=black,fill=green]  {};
\draw (10,2) node[circle,draw=black,fill=green]  {};
\draw (11,1) node[circle,draw=black,fill=white]  {};
\draw (11,0) node[circle,draw=black,fill=white]  {};

\draw (3,14) node[circle, scale=0.6, fill=black]  {};
\draw (3,13) node[circle, scale=0.6, fill=black]  {};
\draw (3,12) node[circle, scale=0.6, fill=black]  {};
\draw (4,11) node[circle, scale=0.6, fill=black]  {};
\draw (5,11) node[circle, scale=0.6, fill=black]  {};
\draw (6,10) node[circle, scale=0.6, fill=black]  {};
\draw (6,9) node[circle, scale=0.6, fill=black]  {};
\draw (6,8) node[circle, scale=0.6, fill=black]  {};
\draw (7,7) node[circle, scale=0.6, fill=red]  {};
\draw (7,6) node[circle, scale=0.6, fill=black]  {};
\draw (7,5) node[circle, scale=0.6, fill=black]  {};
\draw (8,4) node[circle, scale=0.6, fill=black]  {};
\draw (9,3) node[circle, scale=0.6, fill=black]  {};
\draw (10,2) node[circle, scale=0.6, fill=black]  {};
\draw (11,1) node[circle, scale=0.6, fill=black]  {};
\draw (11,0) node[circle, scale=0.6, fill=black]  {};

\draw (6.5,7) node {$u$};
\draw (7.5,6.5) node[red] {$A$};

\end{tikzpicture}
\caption{A planar $\Delta$ (shown with dots) and a corresponding $\OLa$ (shown with circles), with $n=7$. $A$ is shown in red. The construction locates $\bar u$ either among the yellow nodes (if $i=2$) or the green nodes (if $i=1$)}
\label{fig:constr}
\end{figure}

\medskip

\noindent{\bf Case I: 
$o \notin \bar Q$, $u \in o_+$.}

Assume first that $u\in o_+$. Because $u\in\rt$ and $o\in\rt_+$, (ix) of Lemma \ref{lem:terrace} implies that $[o,u]\subset\rt$.  Since we assumed that $o\notin \bar Q$, we must have $u-(n-2)e_j\in[o,u]$ for any $j$. Since $n-1>\rd$,  Lemma \ref{lem:Vd} implies that we can find a $\bar u\in V_d$ of the form $u-ke_j$, $0\le k\le n-2$.  Then properties \eqref{eqn0}--\eqref{eqn5} hold with $\OLa=\rt$.  
Relation \eqref{eqn0} is immediate from \eqref{Lambda=lambda}.  
Relations \eqref{eqn1}--\eqref{eqn6} are trivially true. Relation \eqref{eqn3} is immediate since $o \in \mc{C}^{\sss R}_o(\bso^1)$, and \eqref{eqn4} holds since $-M\1\notin \rt_+$.  Finally, \eqref{eqn5} holds since our $u\in o_+$.

\bigskip

\noindent{\bf Case II:
$o \notin \bar Q$, $u \notin o_+$.}

\noindent {\bf Construction of $\OLa$:}
Assume now that $u\notin o_+$ .
Recall \eqref{Lambda=lambda} and \eqref{Hsuboplus}, i.e.~that 
$\rt=\lambda^{\sss R}_{R \cap \rt_+}$ and $H^{\sss R}_\rt\subset o_+$.
By Lemma \ref{lem:pivotalsitesandcorners} and the fact that $u\notin o_+$, there exists $i\in [d]$ such that $u+e_i\in R\setminus\rt$.  
Without loss of generality, we assume $u+e_1\in R\setminus\rt$. By definition of $\lambda^{\sss R}_{R \cap \rt_+}$ we conclude that 
\begin{equation}
\text{ for each $i\in [d]$, either \quad $u+e_1-e_i\notin R$ \quad or \quad $u+e_1-e_i\in R\cap\rt_+$}.\label{conc}
\end{equation}
Let $I$ be the set of $i$ for which the latter holds.

Define $A=\{u+e_1-e_i: i\in I\}\subset R\cap \rt_+$. Note that $1\in I$ so $u\in A$.  We now push up the corners of $\rt$ in $\bar Q^{\sss R:o}\setminus A$, to give $\rt'=\rt^{\sss{\bar Q,A}}$. By definition $\qt'$ has no corners in $\bar Q^{\sss R:o}\setminus A$.  
Lemma \ref{lem:minimalterrace}(ii) with $G=R \cap \rt_+$ shows that $(\rt')_+=\rt_+$ off $\bar Q^{\sss R:o}$ and by Lemma \ref{lem:minimalterrace} (i) also $\rt'=\rt$ off $\bar Q$ (and therefore also off $Q^{\sss R:o}$ by \eqref{bar_int}).  
This process might lead to the creation of sites in $H^{\sss R}_{\rt'}\setminus o_+$. 
But if it does, any points in that set must lie in $Q^{\sss R:o}$. In other words, 
\begin{equation}
(Q^{\sss R:o})^c\cap H^{\sss R}_{\rt'}\setminus o_+=\varnothing. \label{phew}
\end{equation}

To see this, we will show that if $z\in (Q^{\sss R:o})^c\cap H^{\sss R}_{\rt'}\setminus o_+$ then $z\in  H^{\sss R}_{\rt}\setminus o_+$, and then appeal to the fact that the latter set is empty.  Suppose $z$ is in $H^{\sss R}_{\rt’}$ and also outside $Q^{\sss R:o}$. Then each $z+e_i$ is in $\rt’$ or not in $R$. Suppose $z+e_i$ is in $\rt’$ but not in $\rt$.  Since $\rt = \rt’$ off $\bar Q$, we must have that $z+e_i\in \bar Q\subset (u+Q_{n-2})$. Therefore $z\in u+Q_{n-1}$, so every neighbour of $z$ is either in $Q$ or not in $R$. Therefore $z$ is actually in $Q^{\sss R:o}$, which is a contradiction. In other words, each $z+e_i$ is in $\rt$ or not in $R$. So $z\in H^{\sss R}_{\rt}$ as well.  As noted above this proves \eqref{phew}.

By repeated applications of Lemma \ref{lem:Qterraceconsequences}(i) there would then be corners in $H^{\sss R}_{\rt'}\setminus o_+$, and by \eqref{phew} these corners must be in $Q^{\sss R:o}$.  These corners can be successively deleted (i.e.~pushed up), as in Lemma  \ref{lem:Qterraceconsequences}(iii) without making any changes off $Q^{\sss R:o}$. This gives a new $R$-terrace $\rt''\subset \rt'$, which by construction has no corners in $H^{\sss R}_{\rt''}\setminus o_+$. It follows that $H^{\sss R}_{\rt''}\subset o_+$ (since if not then (as above) repeated applications of Lemma \ref{lem:Qterraceconsequences}(i) would produce corners in $H^{\sss R}_{\rt''}\setminus o_+$).

Now $\rt''$ might have new corners in $\bar Q^{\sss R:o}\setminus A$, but we can push these up in turn to get $\rt'''$, etc. Applying the two above procedures repeatedly, in turn, we eventually stabilize at an $R$-terrace which we take to be $\OLa$. By construction, it has no corners in $\bar Q^{\sss R:o}\setminus A$, and $H^{\sss R}_{\OLa}\subset o_+$.
\bigskip

\noindent {\bf Verification of the Terrace properties:}
Properties \eqref{eqn1} and \eqref{eqn2} follow automatically from our construction. Property \eqref{eqn6} follows, because every time we push up or delete a corner, the region above that terrace gets reduced. 
Property \eqref{eqn0} follows from Remark \ref{rem:lambdacubed}, since \eqref{Qpushupdef} implies that pushing up corners always yields terraces to which that Remark applies.

Property \eqref{eqn4} follows from \eqref{eqn6}, since $S\subset R\setminus\rt_+$. By Lemma \ref{lem:Qterraceconsequences}(ii), pushing up corners in $\bar Q^{\sss R:o}$ will never remove $o$, 
which was assumed to be $\notin \bar Q$.   Therefore \eqref{eqn3} follows as well, since the process of Lemma \ref{lem:Qterraceconsequences}(iii) is never applied to delete $o$ since $o\in o_+$. 
Therefore it only remains to find a $\bar u\in\OLa\cap V_d\cap \bar Q$. Note that since $H^{\sss R}_{\bar\rt}\subset o_+$, any such $\bar u$ will satisfy \eqref{eqn5}.

To do that we will first establish the following properties of the construction: 
\begin{equation}
\label{Aproperties}
\text{$A\subset(\OLa)_+$,  $u+e_1\in R\setminus\OLa$, and $u\in\OLa$.}
\end{equation}

We know that $u\in\rt$, which has the form $\lambda^{\sss R}_G$. So there is an $r\in [d]$ such that $u-e_r\in R\setminus\rt_+$. As in Remark \ref{rem:pushup_subset}, $(\rt^{\sss{\bar Q,A}})_+\subset \rt_+$, therefore $u-e_r\in R\setminus(\rt^{\sss {\bar Q,A}})_+$ too.  
Since $u\in A\subset (\rt^{\sss{\bar Q,A}})_+$ by definition of $\rt^{\sss{\bar Q,A}}$ and Lemma \ref{lem:Qterraceconsequences}(ii) (or Lemma \ref{lem:minimalterrace}(ii)), it follows that $u\in \lambda^{\sss R}_{R\cap (\rt^{\sss{\bar Q,A}})_+}\subset \rt^{\sss{\bar Q,A}}$, the latter by Lemma \ref{lem:QterracepropertiesB}(i).  Now  $u+e_1-e_r\in R$ (since $u+e_1\in R$ and $u-e_r\in R$, and $R$ is a box). So by definition of $A$ and $I$, $u+e_1-e_r\in A\subset (\rt^{\sss{\bar Q,A}})_+$,  
 and because $u-e_r\in  R\setminus(\rt^{\sss{\bar Q,A}})_+$, it follows from Remark \ref{rem:lambdacubed}
that $u+e_1-e_r\in\rt^{\sss{\bar Q,A}}$. 
Since $u+e_1-e_i\in(\rt^{\sss{\bar Q,A}})_+$ for all $i\in I$ (by definition of $A$ and the fact that $A\subset (\rt^{\sss{\bar Q,A}})_+$ as above), we conclude that $u+e_1\in  (\rt^{\sss{\bar Q,A}})_+\setminus \rt^{\sss{\bar Q,A}}$.  To see this, note that if $u+e_1 \in \rt^{\sss{\bar Q,A}}=\lambda^{\sss R}_{R \cap (\rt^{\sss{\bar Q,A} })_+}$ (the latter by Remark \ref{rem:lambdacubed}) then there exists $i_0\in [d]$ with $u+e_1-e_{i_0} \in R\setminus (\rt^{\sss{\bar Q,A}})_+$.  So $i_0\in I$ (by \eqref{conc}) but $u+e_1-e_{i_0} \notin  (\rt^{\sss{\bar Q,A}})_+$, giving a contradiction.

Since $u+e_1\in  R \setminus \rt^{\sss{\bar Q,A}}$ 
we have that  for each $i\in [d]$, $u+e_1-e_i\notin H_{\rt^{\sss{\bar Q,A}}}^{\sss R}$ (so in particular  $u\notin H_{\rt^{\sss{\bar Q,A}}}^{\sss R}$). 
This shows that no point of $A$ will be deleted from $\rt^{\sss{\bar Q,A}}$ by applying the process of Lemma  \ref{lem:Qterraceconsequences} (iii). Pushing up corners in $\bar Q^{\sss R:o}\setminus A$ won't excise any part of $A$ either (Lemma \ref{lem:Qterraceconsequences}(ii)), so in fact, points of $A$ survive every  iteration in the construction of $\OLa$, which verifies that
$A\subset(\OLa)_+$. That implies the second property of \eqref{Aproperties}, since for every $k$, either $u+e_1-e_k\in\OLa$ (for $k \in I$) or it $\notin R$ (so $k \notin I$), so $u+e_1$ can't be in $\OLa$.
The remaining property, that $u\in\OLa$, now follows since $u\in A\subset \OLa_+$  while as noted above, for some $r\in [d]$ we have  $u-e_r\in R\setminus \rt_+\subset R\setminus \OLa_+$.
\bigskip

\noindent {\bf Finding $\bar u$:}
There must exist indexes $i\neq j$ with $u+ne_i\in R$ and $u-ne_j\in R$, since every index has at least one of these properties (as $n<N$), and neither property fails for every index (as $u\notin \1_+$ and $u\notin S$). Therefore $u+ne_i\in Q$ and $u-ne_j\in Q$.

\bigskip

\noindent \emph{Assume first that $i\neq 1$.}  If  $u+(n-3)e_i\in\OLa$ then
 Lemma \ref{lem:QL}  (cf.~(ix) of Lemma \ref{lem:terrace})
  implies that $u+ke_i\in\OLa$ for all $k$ such that $0\le k\le n-3$ (since we already know that $u\in\OLa$). But $n-2\ge \rho_d$ so by Lemma \ref{lem:Vd}
at least one of these points is in $V_d$. We may take that point as $\bar u$ in this case, since all such points also lie in $\bar Q$.

Otherwise $u+(n-3)e_i\in \bar Q\setminus \OLa$, so there is a first $q\in [1,n-3]$ such that $u+q e_i\notin\OLa$.  Set $y_0=u+(q-1) e_i-e_j$ and $y_1=u+qe_i-e_j$. By choice of $i$ and $j$, both $y_0\in \bar Q \subset R$ and $y_1\in \bar Q \subset R$. We claim that at least one of these (which we will then denote $y$) is in $\OLa$.
To see this, note that as $u+qe_i\in R\cap (\OLa)_+\setminus\OLa$ it follows (e.g.~apply the conclusion of Remark \ref{rem:lambdacubed} to the point $u+qe_1\in (\OLa)_+$)
that $y_1\in(\OLa)_+$. If $y_1\in \OLa$ then take $y=y_1$ to establish the claim.  Otherwise $y_1\notin \OLa$, so (again by Remark \ref{rem:lambdacubed} at the point $y_1\in (\OLa)_+$)   $y_0=y_1-e_i\in (\OLa)_+$. But we know that $y_0+e_j=u+(q-1)e_i\in\OLa$, which (by Lemma \ref{lem:QL}) implies that $y:=y_0\in\OLa$, which verifies the claim. %
The assumption that $i\neq 1$ implies that $y\notin A_+$, since every element of $A_+$ is either in $u_+$ or has first component at least $u^{\sss[1]}+1$.
Because $\OLa=\OLa^{\sss{\bar Q,A}}$ (recall the sentence prior to ``verification of the Terrace properties), and $y \in \bar Q$, Lemma \ref{lem:minimalterrace} (iii) provides a $z\in (\OLa)_+\cap (A\cup\partial_{\sss R} \bar Q)$ with $y\in z_+$. Since $y\notin A_+$, in fact $z\in (\OLa)_+\cap\partial_{\sss R} \bar Q$.  Thus $\|z-u\|_{\infty}=n-2$, so there is a $k\in [d]$ such that $|z^{\sss[k]}-u^{\sss[k]}|=n-2$.
Because $y\in\OLa$,  Lemma \ref{lem:QL} (cf.
Lemma \ref{lem:terrace} (ix)) shows that $z_+\cap y_- \subset \OLa$, i.e.~$[z,y]\subset\OLa$.  
We have $z^{\sss [i]}\le y^{\sss [i]}<u^{\sss [i]}+n-2$, $z^{\sss [j]}\le y^{\sss [j]}=u^{\sss [j]}-1<u^{\sss [j]}+n-2$, and for $k\ne i,j$, 
 $z^{[k]}\le y^{[k]}=u^{[k]}$. 
 Now $z^{\sss[k]}-u^{\sss[k]}<n-2$ for each $k$ and $|z^{\sss[k]}-u^{\sss[k]}|=n-2$ for some $k$ implies that 
 $z^{[k]}=u^{[k]}-(n-2)$ for some $k$. But $y^{[k]}\ge u^{[k]}-1$, so $z^{[k]}\le y^{[k]}-(n-3)$. Therefore $y-(n-3)e_k\in[z,y]\subset\OLa$. Arguing as before, there is at least one point from $V_d$ among the $n-2$ points $y-\ell e_i$, $0\le \ell\le n-3$, which we take as $\bar u$. 
\bigskip

\noindent \emph{Special case $i=1$:}
Now turn to the case $i=1$. Let $v=u+e_1-e_j$. By choice of $j$, we know $v\in R$, so in fact $v\in A\subset\OLa$. We will repeat the previous argument, using $v$ instead of $u$.

We know that $v+(n-4)e_i\in R$. If $v+(n-4)e_i\in \OLa$ then we may repeat the earlier argument (now using that $n-3\ge \rho_d$) to see that $v+ke_i\in\OLa\cap V_d\cap\bar Q$ for some $0\le k\le n-4$, which we may then take to be $\bar u$. Otherwise $v+(n-4)e_i\notin \OLa$ and as  before we may find a $q\in[1,n-4]$ and a $y\in\OLa$, with either $y=y_0:=v+(q-1)e_i-e_j$, or $y=y_1:=v+qe_i-e_j$. Equivalently, $y=u+qe_i-2e_j$ or $y=u+(q-1)e_i-2e_j$, which again implies that $y\notin A_{+}$. 

Now $q\le n-4$ implies that $y\in\bar Q$. Arguing as before, there is a $z\in(\OLa)_+\cap\partial_{\sss R} \bar Q$ with $y\in z_+$. Therefore $[z,y]\subset\OLa$. There will now be at least one coordinate $k$ such that $z^{\sss[k]}=u^{\sss[k]}-(n-2)\le y^{\sss[k]}-(n-4)$ and as before, one of the points $y-\ell e_k$, $0\le \ell \le n-4$ will be in $\OLa\cap V_d\cap\bar Q$. We take that point as $\bar u$.
\bigskip

\noindent{\bf Case III: $o\in\bar Q$.}

 Again, there are two possibilities, namely that $u\in o_+$ or $u\notin o_+$. If $u\in o_+$ then as before we have $[o,u]\subset\Delta$. So in fact, $\OLa=\Delta$ and $\bar u=o\in V_d$ satisfy the desired conditions.

Therefore assume $u\notin o_+$. In fact, every step of the previous argument carries through as before, up to the verification of \eqref{eqn3}, which could fail. So consider the procedure we used to obtain $\OLa$, except carry out the excisions or pushing up of corners one by one. If at every stage of this process we have an $R$-terrace $\Delta'$ for which $o\in (\Delta')_+$, then the process can continue to the end, to give $\OLa$ and $\bar u$ satisfying the desired conditions. But if at some stage $o$ ceases to be $\in (\Delta')_+$, we instead stop the process and let $\OLa$ be the last $\Delta'$ with $o\in(\Delta')_+$. Then in fact $o\in\OLa$, as each step of pushing up loses only a point in the current terrace.  Now, every step of the construction process preserves \eqref{eqn0},\eqref{eqn1},\eqref{eqn2},\eqref{eqn6},\eqref{eqn4}, while our stopping rule means that \eqref{eqn3} holds as well.
But also $o\in V_d$ and $o\in o_+$. So in fact, all the desired conditions will hold for this $\OLa$ with the choice $\bar u=o$.
 
This completes the proof.
\end{proof}

\subsection{Orthant model for slabs}
\label{slabs}
Recall the notation introduced after Proposition \ref{prp:strictdisturbedinequality}.

In this section, terraces for $\mathbb{Z}^{d+1}$ will be called {\it $(d+1)$-terraces}, to distinguish them from what we call {\it $d$-terraces}, which will be subsets of $S^{\{i\}}=\Z^d\times \{i\}$ for $i\in [2]$.
For a subset $A^{\{i\}}$ of $S^{\{i\}}$ we will let $(A^{\{i\}})_{+[d]}$ denote the set of points $z\in S^{\{i\}}$ for which there exists $y\in A^{\{i\}}$ such that $z^{\sss[j]}\ge y^{\sss[j]}$ for each $j\in [d]$. We also let $A^{\{i\}}_{+}$ denote the set of points $z\in S^{\{1,2\}}$ such that there exists an $y\in A^{\{i\}}$ such that $z^{\sss[j]}\ge y^{\sss[j]}$ for each $j\in [d+1]$.  It is an easy exercise to see that \begin{equation}
(A^{\{i\}})_{+[d]}=(A^{\{i\}})_+\cap S^{\{i\}}.\label{+d+}
\end{equation}

Exactly as in Proposition \ref{prop:CoStructure} or Lemma \ref{lem:localCoStructure}, one can show that the following holds on the graph $S^{\{1,2\}}$. Loosely, it says that $\mc{C}_{x}^{\{1,2\}}(p)$ is bounded by two $d$-terraces, $\Lambda^1$ in $S^{\{1\}}$ and $\Lambda^2$ in $S^{\{2\}}$, for which the projection of $\Lambda^2$ onto $S^{\{1\}}$ lies below $\Lambda^1$. Both $\Lambda^1$ and $\Lambda^2$ consist of sites in $\Omega_1$, as do any sites in $S^{\{2\}}$ that project onto the region between $\Lambda^1$ and the projection of $\Lambda^2$.
 
\begin{lemma}
\label{lem:slabboundary}
For $p>p_c^{\{1,2\}}(d)$ almost surely for each $x\in S^{\{1,2\}}$ there are $d$-terraces $\Lambda^i\subset S^{\{i\}}$, $i=1,2$ (depending on $x$) such that:
\begin{itemize}
\item[\normalfont(i)] $\Lambda^1\cup\Lambda^2\subset \mc{C}_{x}^{\{1,2\}}(p)$, and 
\item[\normalfont(ii)] $\Lambda^1\cup\Lambda^2\subset \Omega_1$, and 
\item[\normalfont(iii)] $(\Lambda^i)_+\cap S^{\{i\}}=\mc{C}_{x}^{\{1,2\}}(p)\cap S^{\{i\}}$ for each $i$, and 
\item[\normalfont(iv)] $(\Lambda^1)_+\cap S^{\{2\}}\subset (\Lambda^2)_+$, and 
\item[\normalfont(v)] $((\Lambda^2)_+\cap S^{\{2\}})\setminus ((\Lambda^1)_+\cap S^{\{2\}})
\subset \Omega_1$.
\item[\normalfont(vi)] Let $W=e_{d+1}+(\{o\}\cup \mc{E}_-(d))$.
If $z\in\Lambda^1$ then $(z+W)\cap \Omega_1\ne \varnothing$.  
\end{itemize}
\end{lemma}
\begin{proof}
By Lemma \ref{lem:01} (2) and the fact that $p>p_c^{\{1,2\}}(d)$, we have that $\mc{C}_{x}^{\{1,2\}}(p)$ is a.s.~bounded below.
Because $e_j\in \mc{E}_+$ for every $j\in[d+1]$, it is clear that $(\mc{C}_{x}^{\{1,2\}}(p))_+=\mc{C}_{x}^{\{1,2\}}(p)$. Therefore $\mc{C}_{x}^{\{1,2\}}(p)\cap S^{\{i\}}$ is a \emph{good} set (viewed as a subset of $\Z^d$), and (recall Definitions \ref{def:good} and \ref{def:terrace}) we may define $\Lambda^i=\lambda_{\mc{C}_{x}^{\{1,2\}}(p)\cap S^{\{i\}}}\subset S^{\{i\}}$.

 Claim (i) is immediate.   We have (ii) because if $x\in\Lambda^i\setminus\Omega_1$, then $x-e_j$ would $\in \mc{C}_{x}^{\{1,2\}}(p)\subset (\Lambda^i)_{+[d]}$ for every $j\in[d]$, contradicting the fact that $\Lambda^i$ is a $d$-terrace. 
 Claim (iii) holds since $(\Lambda^i)_+\cap S^{\{i\}}=(\Lambda^i)_{+[d]}=\mc{C}_{x}^{\{1,2\}}(p)\cap S^{\{i\}}$ (the first equality from \eqref{+d+}, and the second quantity by Lemma \ref{lem:terrace}(ii) applied to $S^{\{1\}}$ as a copy of $\Z^d$).  Since $(\Lambda^1)_+\cap S^{\{2\}}\subset \mc{C}_{x}^{\{1,2\}}(p)\cap S^{\{2\}}\subset (\Lambda^2)_+$, (the latter inclusion by (iii)),  (iv) follows.

To prove (v), suppose that $y \in ((\Lambda^2)_+\cap S^{\{2\}})$.  Then by (iii) $y \in \mc{C}_{x}^{\{1,2\}}(p)\cap S^{\{2\}}$.  
If $y-e_{d+1}\in \mc{C}_{x}^{\{1,2\}}(p)$ then $y-e_{d+1}\in \mc{C}_{x}^{\{1,2\}}(p)\cap S^{\{1\}}=(\Lambda^1)_+\cap S^{\{1\}}$, thus $y \in (\Lambda^1)_+\cap S^{\{1,2\}}$.  It follows that if $z\in ((\Lambda^2)_+\cap S^{\{2\}})\setminus ((\Lambda^1)_+\cap S^{\{2\}})$ then $z\in \mc{C}_x^{\{1,2\}}(p)$ but $z-e_{d+1}\notin \mc{C}_{x}^{\{1,2\}}(p)$, so $z\in \Omega_1$.  
To prove (vi), let $z\in\Lambda^1$. First note that if $z+e_{d+1}\in\Lambda^2$ then by (ii) we are done. So assume otherwise.  Since $z+e_{d+1}\in (\Lambda^1)_+\cap S^{\{2\}}\subset (\Lambda^2)_+$ (the latter by (iv)) we have that $z+e_{d+1}\in  (\Lambda^2)_+\setminus  \Lambda^2$ so $z+e_{d+1}-e_j\in(\Lambda^2)_+$ for each $j\in[d]$.  Because $\Lambda^1$ is a $d$-terrace, there is a $k\in[d]$ such that $z-e_k\in S^{\{1\}}\setminus (\Lambda^1)_{+[d]}$. Since $(\Lambda^1)_{+[d]}=(\Lambda^1)_+\cap S^{\{1\}}$ (by \eqref{+d+}) it follows that also $z-e_k\notin(\Lambda^1)_{+}$, so  $z-e_k+e_{d+1}\in ((\Lambda^2)_+\cap S^{\{2\}})\setminus ((\Lambda^1)_+\cap S^{\{2\}})$. Therefore by (v), $z-e_k+e_{d+1} \in \Omega_1$.
\end{proof}

\begin{proof}[Proof of Proposition \ref{slabstrictinequality}]
We show that whenever $p>p_c^{\{1,2\}}(d)$ we have $p\ge p_{c}(f,d)$. 

Let $\bs {\mc{G}}=(\mc{G}_x)_{x\in\mathbb{Z}^{d+1}}$ be the half-orthant model in dimension $d+1$. As we have seen above $(\mc{G}_x)_{x\in S^{\{1\}}}$  (using only arrows that stay in $S^{\{1\}}$) realizes the half-orthant model in dimension $d$, via the natural correspondence between $S^{\{1\}}$ and $\mathbb{Z}^d$. We will now define two other environments $\bs{\eta}=(\eta_x)_{x\in S^{\{1\}}}$  [resp. $\bs{\zeta}=(\zeta_x)_{x\in S^{\{1\}}}$] by specifying the sets $\Omega_1({\bs \eta})$ [resp. $\Omega_1(\bs{\zeta})$] upon which $\eta_x=\mc{E}_+$ [resp. $\zeta_x=\mc{E}_+$].

Let $W$ be defined as in Lemma \ref{lem:slabboundary}(vi).  For $V_d$ as in Section \ref{disturbedorthant}, set $\hat E_1=V_d\times\{1\}\subset\mathbb{Z}^{d+1}$. For $x \in S^{\{1\}}$ we set 
\begin{align}
&x\in \Omega_1(\bs{\eta}) \iff \big\{x\in \Omega_1(\bs{\mc{G}}) \text{ and }(x+W) \cap \Omega_1(\bs{\mc{G}})\ne \varnothing\big\}, \text{ and }\label{eta}\\
&x \in \Omega_1(\bs{\zeta}) \iff \big\{ x\in \hat E_1 \cap \Omega_1(\bs{\eta}) \quad \text{ or } \quad x\in (S^{\{1\}}\setminus \hat E_1)\cap \Omega_1(\bs{\mc{G}})  \big\}.\label{zeta}
\end{align}
Now note that (with $(U_z)_{z \in \Z^{d+1}}$ i.i.d.~standard uniforms) for $x\in S^{\{1\}}$, $\eta_x=\eta_x(p)$ depends on $U_x$, $U_{x+e_{d+1}}$ and $\{U_{x+e_{d+1}-e_j}:j \in [d]\}$.  For $x\in \hat{E}_1\subset S^{\{1\}}$, $\zeta_x(p)$ depends on the same variables.  For $x\in S^{\{1\}}\setminus \hat{E}_1$, $\zeta_x(p)$ depends on $U_x$ only.  Thus if $x,x'\in S^{\{1\}}$ are distinct then $\zeta_x$ and $\zeta_{x'}$ depend on disjoint collections of $U_z$ (this is obvious unless $x,x'\in \hat{E}_1$, in which case we use the definition of $V_d$ and Lemma \ref{lem:Edspreadout}). 

 It is now easy to check that in fact $\bs{\zeta}$  is an $f$-disturbed orthant model. In particular for $x\in \hat{E}_1$,
\begin{align*}
\P(\zeta_x(p)=\mc{E}_+)&=\P(U_x\le p)\Big(1-\P(U_{x+e_{d+1}}>p)\prod_{i=1}^d\P(U_{x+e_{d+1}-e_i}>p)\Big)\\
&=p(1-(1-p)^{d+1}).
\end{align*}

Let $p>p_c^{\{1,2\}}(d)$, and let $\Lambda^1\subset \Omega_1(\bs{\mc{G}})$ be the $d$-terrace given by Lemma \ref{lem:slabboundary} (using only arrows in $S^{\{1,2\}}$).   
By Lemma \ref{lem:slabboundary}(vi), $\Lambda^1\subset \Omega_1(\bs{\eta})$. 
Moreover, it is easy to see from \eqref{eta} and \eqref{zeta} that $\Omega_1(\bs{\eta})\subset \Omega_1(\bs{\zeta})$. Therefore $\Lambda^1$ is also a $d$-terrace of $\Omega_1$ sites for $\bs{\zeta}$. This shows that  $\mc{C}_o(p,f(p))\subset \Z^d$ is a.s.~bounded below when $p>p_c^{\{1,2\}}(d)$.  By Lemma \ref{lem:01}(1) this means that for $p>p_c^{\{1,2\}}(d)$ we have that $\P(\mc{C}_o(p,f(p))=\Z^d)=0$, so $p\ge p_c(f,d)$ as claimed .
\end{proof}

\section*{Acknowledgements}
Salisbury's work was supported in part by a grant from NSERC. The work of Holmes was supported in part by Future Fellowship FT160100166 from the Australian Research Council.  MH thanks the Pacific Institute for the Mathematical Sciences for hosting part of this research.  The authors sincerely thank Geoffrey Grimmett for several helpful comments, including identifying some inaccuracies in the introduction.

\bibliographystyle{plain}

\end{document}